\documentclass[11pt, oneside]{amsart}   	
\usepackage{geometry}                		
\usepackage{graphicx}				
								
\usepackage{hyperref}								
\usepackage{amssymb}
\usepackage{amsmath}
\usepackage{float}
\usepackage{xcolor}
\usepackage{tikz}
\usetikzlibrary{calc}
\usepackage{caption}
\usepackage{booktabs}
\usepackage{varwidth}

\usepackage{amsthm}
\usepackage{amsmath}
\usepackage{mathtools}
\usepackage{amssymb}
\usepackage{commath}
\usepackage{bigints}
\usepackage{romanbar}
\usepackage{algorithm2e}
\usepackage{xcolor}
\usepackage{enumitem}
\setlist{leftmargin=*}

\newtheorem{theorem}{Theorem}
\newtheorem{prop}{Proposition}
\newtheorem{cor}{Corollary}
\newtheorem{remark}{Remark}

\newtheorem{lemma}{Lemma}[section]
\newcommand{\vc}{\mathbf{c}}
\newcommand{\ve}{\mathbf{e}}
\newcommand{\vg}{\mathbf{g}}
\newcommand{\vx}{\mathbf{x}}

\newcommand{\vu}{\mathbf{u}}

\newcommand{\vv}{\mathbf{v}}

\newcommand{\vf}{\mathbf{f}}
\newcommand{\vn}{\mathbf{n}}

\newcommand{\vphi}{\boldsymbol{\phi}}

\newcommand{\va}{\mathbf{a}}
\newcommand{\vb}{\mathbf{b}}

\newcommand{\vH}{\mathbf{H}}

\newcommand{\Th}{\mathcal{T}_h} 
\newcommand{\V}{\mathbb{V}} 

\newcommand{\veps}{\varepsilon}

\newcommand{\Etotal}{E_{\textrm{total}}}
\newcommand{\Emag}{E_{\textrm{mag}}}

\newcommand{\tr}{{\rm tr}} 
\newcommand{\curl}{{\rm curl}\,} 
 
\newcommand{\di}{\mathop{\rm div}\nolimits} 

\newcommand{\Cstab}{C_\textrm{stab}}

\newcommand{\eps}{\varepsilon}

\renewcommand{\norm}[1]{\|#1\|}
\newcommand\divrg{\mathrm{div}\,}

\title{Projection-Free Method for the Full Frank-Oseen Model of Liquid Crystals}
\author{Lucas Bouck, Ricardo H.\ Nochetto}

\begin{document}

\begin{abstract}
Liquid crystals are materials that experience an intermediate phase where the material can flow like a liquid, but the molecules maintain an orientation order. The Frank-Oseen model is a continuum model of a liquid crystal. The model represents the liquid crystal orientation as a vector field and posits that the vector field minimizes some elastic energy subject to a pointwise unit length constraint, which is a nonconvex constraint. Previous numerical methods in the literature assumed restrictions on the physical constants or had regularity assumptions that ruled out point defects, which are important physical phenomena to model. We present a finite element discretization of the full Frank-Oseen model and a projection free gradient flow algorithm for the discrete problem in the spirit of Bartels (2016). We prove $\Gamma$-convergence of the discrete to the continuous problem: weak convergence of subsequences of discrete minimizers and convergence of energies. We also prove that the gradient flow algorithm has a desirable energy decrease property. Our analysis only requires that the physical constants are positive, which presents challenges due to the additional nonlinearities from the elastic energy.
\end{abstract}

\maketitle

\section{Introduction}

Liquid crystals are materials that experience an intermediate phase of matter between solid and fluid. In the nematic phase, they often retain an orientation order but fail to retain a positional order. They may also react easily to external fields. These properties lead to the use in optical applications \cite{schadt1997liquid, kim2020tunable}. For stationary continuum models of liquid crystals, the main models include the Frank-Oseen model \cite{oseen1933theory, frank1958liquid}, Ericksen model \cite{ericksen1991liquid}, and Landau de-Gennes or $Q$ tensor model \cite{de1993physics}. Two books are dedicated to mathematical modeling of liquid crystals \cite{de1993physics,virga1995variational}. This paper deals with the full Frank-Oseen model. Two numerical challenges are the nonconvex pointwise unit length constraint and the quartic nonlinear structure. This paper addresses these challenges by applying a projection-free gradient flow in the spirit of \cite{bartels2016projection}, and explicit treatment of nonlinearities as in \cite{bartels2022stable}, to derive an energy stable gradient flow for the full Frank-Oseen model. This paper also proves $\Gamma$-convergence, thereby extending results for harmonic maps \cite{bartels2015numerical}.

\subsection{Frank-Oseen model}

The Frank-Oseen Model \cite{oseen1933theory, frank1958liquid} is a continuum model of a nematic liquid crystal occupying a bounded domain $\Omega\subset \mathbb{R}^3$. The model represents the liquid crystal with a director field $\vn: \Omega \to  \mathbb{S}^2:=\{ \vv\in \mathbb{R}^3 : |\vv|=1\}$. At a point $x\in \Omega$, the unit length vector $\vn(x)$ describes the average orientation of the liquid crystal molecules. The Frank-Oseen model \cite{frank1958liquid, oseen1933theory} posits  that $\vn$ minimizes the following elastic energy:
\begin{equation*}
E_{FO}[\vn] = \frac{1}{2}\int_\Omega k_1(\text{div} \vn)^2 + k_2(\vn\cdot \text{curl} \vn)^2 +k_3|\vn\times \text{curl} \vn|^2 + (k_2+k_4)\left(\text{tr}((\nabla \vn)^2) - (\text{div} \vn)^2\right) d\vx
\end{equation*}
over the admissible class of functions $H^1(\Omega;\mathbb{S}^2)$.
The four constants $k_i$ are known as Frank's constants, and correspond to splay ($k_1>0$), twist ($k_2>0$,
  bend ($k_3>0$), and saddle-splay ($k_2+k_4>0$). Ericksen showed that for $E_{FO}[\vn]$ to be positve semi-definite,
the Frank's constants must obey the following relations \cite{ericksen1966,virga1995variational}
\[
2k_1 \ge (k_2+k_4),
\quad
k_2 \ge |k_4|,
\quad
k_3 \ge 0,
\]
Ericksen also showed in \cite{ericksen1962nilpotent} that the saddle-splay only depends on the trace of $\vn$ on
$\partial\Omega$.
Previous analytical work \cite{hardt1986existence} proved existence of minimizers of $E$ over the admissible set
\begin{equation*}
\mathcal{A}_{\vg}:= \{ \vn\in H^1(\Omega;\mathbb{R}^3) : |\vn|=1 \text{ a.e. in } \Omega \text{ and } \vn|_{\partial\Omega} = \vg\},
\end{equation*}
where $\vg:\partial\Omega\to \mathbb{S}^2$ is Lipschitz and $k_i>0$ for $i=1,2,3$. The main idea for proving existence of minimizers is to write a modified but equivalent energy with modified coefficients $c_0>0$ and $c_i\geq 0$ for $i=1,2,3$, because the saddle-splay depends on $\vg$ but not on $\vn$:
 \begin{equation*}
E [\vn] = \frac{1}{2}\int_\Omega c_0|\nabla \vn|^2 + c_1(\di \vn)^2 + c_2(\vn\cdot \curl \vn)^2 +c_3|\vn\times \curl \vn|^2 d\vx. 
\end{equation*}
We recall the relevant results and observations from the analysis in Section \ref{sec:prelim}.

\subsection{Previous related numerical works}\label{sec:prev-numerical-frank}
  
There are many numerical methods to compute minimizers to the Frank-Oseen Energy \cite{adler2015energy, alouges1997new, alouges1997minimizing, bartels2005stability, bartels2016projection, bartels2022quasi, bartels2022benchmarking, cohen1989relaxation, hu2009saddle, hu2014newton, xia2021augmented}. The collection of works \cite{alouges1997new, alouges1997minimizing, bartels2005stability} use a type of steepest descent method. At each step, the steepest descent only searches in tangent directions to linearize the unit length constraint. The violation of the constraint is then controlled by projecting the solution to satisfy the unit length constraint. This projection step often requires {\it weakly acute meshes} to guarantee energy decrease, which creates challenges for meshing in 3D as documented in \cite{nochetto2017finite}.
Other gradient flow methods, so-called projection-free methods, use pseudotime-step parameter to control the constraint violation \cite{bartels2016projection, hu2014newton}. The previous projection-free works have also dealt with simplifications of the Frank-Oseen energy such as the one-constant approximation (also known as harmonic maps) in \cite{bartels2016projection} or $k_2=k_3$ in \cite{hu2014newton}. We discuss these simplifications in more detail in Remark \ref{rmk:simplications-of-E} below.

The other group of methods use a Lagrange multiplier to enforce the unit length constraint \cite{adler2015energy,bartels2022quasi, hu2009saddle}. One advantage of using a Lagrange multiplier is that one can use a Newton method to solve the discrete problem. The analysis of these methods typically require additional regularity on the solution, which excludes point defects. We refer to Remark \ref{rmk:regularity} below for related discussion. Finally, we make note of the Newton type methods explored in these works. The work \cite{adler2015energy} discretizes the system for solving a Newton iteration of the full Frank-Oseen energy and presents an error analysis of the Newton linearization.  Also, the recent work \cite{bartels2022quasi} proved quasioptimal error estimates for a finite element discretization of harmonic maps as long as the solution is regular and stable. The theory in \cite{bartels2022quasi} also suggests super-linear convergence of Newton method if the initial guess is sufficiently close. The related work \cite{bartels2022benchmarking} computationally explores the performance of various optimization methods and discretizations for harmonic maps. 

We also note that there are many other works on numerical methods for nematic liquid crystals; we refer to the review \cite{wang2021modelling} and references therein. Some works we would like to highlight include projection methods for the Ericksen and Q tensor models \cite{nochetto2017finite, borthagaray2020structure}, a projection-free method for the Ericksen model \cite{nochetto2022gamma} and other works on Ericksen and Q-tensor models \cite{walker2020finite, schimming2021numerical, davis1998finite}. 
It should be noted that \cite{walker2020finite} proves $\Gamma$-convergence for the full Ericksen model, while the focus of this paper is on the full Frank-Oseen model, for which such an analysis has been absent in the literature.

\subsection{Our contribution}

This work presents a numerical method for computing minimizers of the full Frank-Oseen energy, without additional restrictions on the elastic constants or triangulation. The method computes minimizers of the modified energy $E$ over the following discrete admissible set:
$$
\mathcal{A}_{\vg, h,\eta} := \Big\{\vv_h\in \mathbb{V}_h: \big\|I_h[|\vv_h|^2-1]\big\|_{L^1(\Omega)}\leq \eta, \big\| \vv_h|_{\partial\Omega} -\vg\big\|_{L^2(\partial\Omega;\mathbb{R}^3)}\leq \eta, \Vert \vv_h\Vert_{L^\infty(\Omega;\mathbb{R}^3)}\leq C \Big\},
$$
where $\mathbb{V}_h$ is the space of continuous piecewise linear vector-valued functions on a triangulation $\mathcal{T}_h$ with mesh size $h$, $\mathcal{N}_h$ is the set of nodes of $\mathcal{T}_h$, and ${I}_h$ is the Lagrange interpolation operator over $\mathcal{T}_h$. Also, $C>1$ is a constant that does not need to change with $h$. Our contributions are as follows.
In Section \ref{sec:discrete-prob}, we prove that the discrete minimization problem $\Gamma$-converges to the continuous problem and prove that discrete minimizers converge up to a subsequence to a minimizer of the continuous problem. The analysis only relies on $k_i>0$, which is what is required by existence \cite{hardt1986existence}. Our analysis extends the $\Gamma$-convergence analysis of \cite[Example 4.6]{bartels2015numerical} for harmonic maps, i.e. $k_1=k_2=k_3$. In Section \ref{sec:proj-free}, we propose a  projection-free gradient flow algorithm to compute critical points of $\tilde{E}$ over $\mathcal{A}_{\vg,h,\eta}$ inspired by \cite{bartels2016projection,bartels2022stable}. Following work first done in \cite{bartels2022stable} and later by \cite{bonito2023gamma} in the context of bilayer plates, we use a linear extrapolation of $(\vn\cdot \curl \vn)^2$ and $|\vn\times \curl \vn|^2$ at every step of the gradient flow. As a result, the gradient flow only requires solving linear systems. Under a mild condition $\tau h^{-1}\leq \Cstab$, where $\Cstab>0$ depends on $k_i$ and the initial data, the gradient flow is energy stable and provides control of the violation of the unit length constraint at nodes in terms of $\tau h^{-1}$. We extend our results in Section \ref{sec:magnetic} to account for a fixed magnetic field. Finally in Section \ref{sec:computation}, we present computational experiments. We highlight quantitative properties of the algorithm as well as the effects of Frank's constants on defect configurations and of an external magnetic field on LC configurations.

\section{Notation and preliminaries}\label{sec:prelim}

Since the $L^2(\Omega)$ norm and inner products are used frequently in this paper, a norm should be assumed to be $L^2$ unless otherwise specified. For $u, v\in L^2(\Omega)$, the $L^2$-inner product will be denoted by $(u,v) = \int_\Omega uv \, d\vx$ and the corresponding norm by $\norm{u} = \sqrt{(u,u)}$. We further shorten notation by denoting the Sobolev norm of a function $u$ as $\norm{u}_{W^{k,p}(\Omega)} = \norm{u}_{W^{k,p}}$ when the domain of integration is clearly $\Omega$.  

\subsection{Properties of the full Frank-Oseen model}

We begin by setting notation and summarizing the results in \cite{hardt1986existence} for the continuous problem.
We first recall the full Frank-Oseen energy, containing splay, twist, and bend as well as saddle-splay energies:
\begin{equation}\label{E:FO}
\begin{aligned}
  E_{FO}[\vn] = \frac{1}{2}\int_\Omega k_1(\divrg \vn)^2 &+ k_2(\vn\cdot {\curl} \vn)^2 +k_3|\vn\times {\curl} \vn|^2
  \\&
  +(k_2+k_4)\big(\tr((\nabla \vn)^2) - (\divrg \vn)^2\big) d\vx. 
\end{aligned}
\end{equation}
We then define the admissible set of director fields as 
\begin{equation}
\mathcal{A}_\vg := \big\{ \vn \in H^1(\Omega;\mathbb{S}^2) : \quad \vn|_{\partial \Omega} = \vg \big\}.
\end{equation}
We assume that $\vg$ is Lipschitz, which implies that $\mathcal{A}_\vg$ is nonempty \cite[Lemma 1.1]{hardt1986existence}.

Every term in $E$ is not problematic from the computational point of view except for the saddle-splay term $\tr((\nabla \vn)^2) - (\divrg \vn)^2$. At first glance, it is not entirely clear that this term is even bounded from below. This poses challenges to both proving existence of minimizers and computation. However, in the presence of Dirichlet boundary conditions, \cite{ericksen1962nilpotent} and \cite[Lemma 1.1]{hardt1986existence} prove that the saddle splay term is constant.

\begin{lemma}[saddle splay]\label{prop:saddle-splay}
There exists a constant $C_\vg$ such that for all $\vn \in \mathcal{A}_\vg$, we have 
\begin{equation}\label{eq:saddle-splay}
C_\vg = \int_\Omega \big(\tr((\nabla \vn)^2) - (\divrg \vn)^2\big) d\vx
\end{equation}
\end{lemma}
The proof of this lemma relies on showing that the saddle splay can be written as a divergence, which means that its contribution only depends on boundary data; this property was first realized by Ericksen \cite{ericksen1962nilpotent}. In the presence of Dirichlet boundary conditions, this means that the saddle splay energy is constant and solely depends on $\vg$ and $\partial\Omega$.

Lemma \ref{prop:saddle-splay} is critical to modify the energy $E_{FO}$ by adding a multiple of $C_\vg$ without changing the minimizers. This leads to the following modified energy $E$ \cite[Corollary 1.3]{hardt1986existence}.
\begin{prop}[modified energy]\label{prop:modified-energy}
Let $c_0 = \min_{i=1,2,3}\{k_i\}>0$ and let $c_i = k_i - c_0\geq 0$. Define $E: \mathcal{A}_\vg\to \mathbb{R}$ by 
\begin{equation}\label{eq:tilde-E}
E[\vn] := {E}[\vn] + \frac{1}{2} \big(c_0 - k_2 - k_4 \big)C_\vg.
\end{equation}
Then, $\vn^*\in \mathcal{A}_\vg$ is a minimizer of $E$ in $\mathcal{A}_\vg$ if and only if $\vn^*$ is a minimizer of $E_{FO}$ in $\mathcal{A}_\vg$.
\end{prop}

It is clear that, since $C_\vg$ is a constant, the minimizers of $E$ are also minimizers of $E_{FO}$. However, the explicit form of $E$ is not readily amenable to computation. Below is a proposition that states an explicit form of $E$ \cite{hardt1986existence}, which we prove for completeness.

\begin{prop}[explicit form of $E$]
Let $c_0 = \min_{i=1,2,3}\{k_i\}>0$ and let $c_i = k_i - c_0\geq 0$ for $i=1,2,3$. Then, for $\vn\in H^1(\Omega;\mathbb{S}^2)$ there holds
\begin{equation}\label{eq:E-tilde2}
E[\vn] = \frac{1}{2}\int_\Omega c_0|\nabla \vn|^2 + c_1(\divrg \vn)^2 + c_2(\vn\cdot {\curl} \vn)^2 +c_3|\vn\times {\curl} \vn|^2 d\vx.
\end{equation}
\end{prop}
\begin{proof}
 Using the expression for $C_\vg$ in \eqref{eq:saddle-splay}, we rewrite $E$ in \eqref{eq:tilde-E} as
\begin{align*}
E[\vn] = \frac{1}{2}\int_\Omega c_0\tr((\nabla\vn)^2) +(k_1 - c_0)(\divrg \vn)^2 +k_2(\vn\cdot \curl\vn)^2+k_3|\vn\times \curl\vn|^2d\vx.
\end{align*}
Since $|\vn| = 1$ a.e.\ in $\Omega$, $(\vn\cdot \curl\vn)^2+|\vn\times \curl\vn|^2 = |\curl \vn|^2$. Hence, adding and subtracting $c_0|\curl \vn|^2$ to $E[\vn]$ yields
\begin{align*}
E[\vn] = \frac{1}{2}\int_\Omega c_0\left[\tr((\nabla\vn)^2)+|\curl \vn|^2\right] +&(k_1 - c_0)(\divrg \vn)^2 \\
+&(k_2-c_0)(\vn\cdot \curl\vn)^2+(k_3-c_0)|\vn\times \curl\vn|^2d\vx.
\end{align*}
The assertion follows from the fact that $|\nabla \vn|^2 = \tr((\nabla\vn)^2)+|\curl \vn|^2$ and the definition of $c_i = k_i - c_0$ for $i=1,2,3$.
\end{proof}
The modified energy $E$ immediately looks friendlier than $E_{FO}$. First, it is easy to tell that $E$ is bounded from below. Secondly, $E$ is coercive in $H^1$ because $c_0>0$. Thirdly, $E$ is weakly lower semicontinuous in $\mathcal{A}_\vg$ because each $c_i\geq0$. These lead to existence of minimizers of $E$. These facts are proved in \cite[Lemma 1.4, Theorem 1.5]{hardt1986existence} and are summarized by the following Lemma.
\begin{lemma}[properties of $E$]\label{lem:equicoercive} The modified energy $E$ is w.l.s.c.\ in $H^1(\Omega;\mathbb{S}^2)$ and 
\begin{equation}\label{eq:equicoercivity-estimate}
\frac{1}{2}c_0\int_\Omega |\nabla \vn|^2d\vx\leq E[\vn] \leq 3(k_1+k_2+k_3)\int_\Omega |\nabla \vn|^2d\vx
\end{equation}
for all $\vn\in H^1(\Omega;\mathbb{S}^2)$. Moreover, there exists a minimizer of $E$ over the admissible set $\mathcal{A}_\vg$.
\end{lemma}

\begin{remark}[modified energy $E$]\label{rem:wlsc}
\rm  
The proof of the weak lower semicontinuity only relies on the fact that $c_i\geq 0$. Also the coercivity only relies\ on $c_0>0$. Hence the coercivity and weak lower semicontinuity of $E$ defined in \eqref{eq:E-tilde2} hold over the space $H^1(\Omega; \mathbb{R}^3)$, which is larger than $H^1(\Omega; \mathbb{S}^2)$. Thus, we will compute with $E$ as defined in \eqref{eq:E-tilde2}.
\end{remark}

\begin{remark}[simplifications of $E$]\label{rmk:simplications-of-E}
\rm
If $k_1 = k_2 = k_3 = 1$, then $E$ takes on the form
\begin{equation*}
E[\vn] = \frac{1}{2}\int_\Omega|\nabla \vn|^2d\vx,
\end{equation*}
which is known as the {\it one constant} approximation. Also, if $k_2 = k_3 > k_1$, then $E$
becomes
 \begin{equation*}
E[\vn] = \frac{1}{2}\int_\Omega k_1|\nabla \vn|^2 + c_2|\curl \vn|^2d\vx,
\end{equation*}
which was studied in \cite{glowinski2003operator,hu2014newton}.
\end{remark}

\subsection{Discretization}

We first define some notations for the discrete problem and summarize some useful results. We consider a sequence of quasiuniform, shape-regular triangulations $\Th$ of $\Omega$. The set of nodes of $\Th$ is denoted by $\mathcal{N}_h$. The space of continuous piecewise linear vector fields is defined by
\begin{equation*}
\mathbb{V}_h := \big\{\vv_h\in C^0(\Omega;\mathbb{R}^3) : \vv_h|_T\in \mathcal{P}_1 \quad \forall T\in \Th \big\}.
\end{equation*}
Similarly, $\mathbb{Q}_h$ denotes the space of continuous piecewise linear real-valued functions:
\begin{equation*}
\mathbb{Q}_h := \big\{v_h\in C^0(\Omega) : v_h|_T\in \mathcal{P}_1 \quad \forall T\in \Th \big\}.
\end{equation*}
We also set $\mathbb{V}_{h,0} := \big \{\vv_h\in \mathbb{V}_h : \vv_h|_{\partial\Omega} = 0 \big\}$ (resp. $\mathbb{Q}_{h,0}$) to be the discrete space of vector-valued fields (resp. scalar fields) with zero boundary conditions.

Another space that will be useful in the {\it gradient flow} algorithm is the space $\mathbb{T}_h(\vn_h)$ of tangent directions to $\vn_h$ at nodes, namely
\begin{equation*}
\mathbb{T}_h(\vn_h) := \big\{\vv_h\in\mathbb{V}_{h,0}: \vv_h(z)\cdot \vn_h(z) = 0 \quad \forall z\in \mathcal{N}_h\big\}.
\end{equation*}
Additionally, given a pseudotime-step $\tau>0$, we let the discrete time derivative be the backward difference:
\begin{equation*}
d_t\vn^{k+1}_h := \frac{1}{\tau}\left(\vn^{k+1}_h - \vn^{k}_h\right).
\end{equation*}

We now state two useful results without proof that are needed for the numerical method. The first result is a Corollary of \cite[Lemma 2.1]{bartels2016projection}.
\begin{lemma}[discrete unit length constraint]\label{lem:unit-length-bartels}
Let $\vn_h$ be a uniformly bounded sequence in $H^1(\Omega;\mathbb{R}^3)$ and further suppose $\vn_h\to \vn$ strongly in $L^2(\Omega;\mathbb{R}^3)$. If $\lim_{h\to0}\norm{I_h[|\vn_h|^2 - 1]}_{L^1} = 0$, then $|\vn| = 1$ a.e. in $\Omega$.
\end{lemma}

We next state a discrete Sobolev inequality that connects the $L^\infty$-norm and $H^1$-norm. This result is an easy consequence of a global inverse inequality and Sobolev imbedding and is well known \cite[Remark 3.8]{bartels2015numerical}.

\begin{lemma}[discrete Sobolev inequality]\label{lem:discrete-sob}
Let $\vv_h\in \mathbb{V}_{h,0}$. There is a constant $c_{inv}$ independent of $h$ such that for all $\vv_h\in \mathbb{V}_{h,0}$: 
\begin{equation*}
\norm{\vv_h}_{L^\infty}\leq c_{inv} h^{-1/2}\norm{ \nabla \vv_h}.
\end{equation*}
\end{lemma}

\section{Discrete minimization problem}\label{sec:discrete-prob}

The discrete minimization problem mimics the continuous problem. The main differences are that rather than enforcing the constraint $|\vn_h| = 1$ pointwise, which would lead to locking, the constraint is enforced at the nodes of the mesh $\Th$ and relaxed by a parameter $\eta>0$. The discrete admissible set is then
\begin{equation}\label{eq:discrete-admin-set}
\mathcal{A}_{\vg, h,\eta} := \big\{\vv_h\in \mathbb{V}_h: \big\|I_h[|\vv_h|^2-1]\big\|_{L^1}\leq \eta,\;\big\| \vv_h|_{\partial\Omega} -\vg\big\|_{L^2(\partial\Omega;\mathbb{R}^3)}\leq \eta,\;  \Vert \vv_h\Vert_{L^\infty}\leq C \big\}.
\end{equation}
We note that $C>1$ is a fixed constant; we only need a uniform $L^\infty$ bound of $\vv_h$ rather than an $L^\infty$ control of the constraint. The parameter $\eta=\eta_h$ satisfies $\eta_h\to0$ as $h\to0$.

The discrete problem is to find $\vn_{h,\eta}$ such that 
\begin{equation*}
\vn_{h,\eta} \in \text{argmin}_{\vv_h\in \mathcal{A}_{\vg, h,\eta}} E[\vv_h].
\end{equation*}
The next task is to prove convergence of the discrete minimizers.

\subsection{ Convergence of minimizers}

The framework follows that of $\Gamma$-convergence. Recall that $\mathcal{A}_\vg$ is nonempty if $\vg$ is Lipschitz. We first construct a recovery sequence.

\begin{lemma}[recovery sequence] \label{lem:recovery-sequence}
Let $\vn\in \mathcal{A}_\vg$. There exists a sequence $\eta_h=C(\Omega) h^{1/2} \|\vn\|_{H^1}\to0$ and $\vn_{h}\in \mathcal{A}_{\vg,h,\eta_h}$ such that $\vn_{h}\to \vn$ in $H^1(\Omega;\mathbb{R}^3)$ and $E[\vn_{h}]\to E[\vn]$ as $h\to0$.
\end{lemma}
\begin{proof}
Let $\vn\in \mathcal{A}_\vg \neq \emptyset$. We proceed in three steps.

{\it 1. Approximation}: Let $\vn_h = \mathcal{I}_h\vn$ be a Cl\'{e}ment interpolant of $\vn$,
i.e. $\vn_h$ is defined by
$$
\vn_h := \sum_{z\in \mathcal{N}_h}\vn_z\phi_z, \quad \vn_z := |\omega_z|^{-1}\int_{\omega_z}\vn,
$$
where $\{\phi_z\}_{z\in \mathcal{N}_h}$ is the nodal basis of $\mathbb{V}_h$, and $\vn_z$ is the average of $\vn$ over the patch $\omega_z$.

We have that $\vn_h\to \vn$ in $H^1(\Omega;\mathbb{R}^3)$ and the $L^2$-error estimate $\|\vn_h-\vn\|\lesssim h \|\vn\|_{H^1}$ holds. We also have the uniform $L^\infty$ bound 
$$
\Vert \vn_h\Vert_{L^\infty}\leq \Vert \vn\Vert_{L^\infty}=1.
$$
Applying the standard trace inequality, there is a constant $C=C(\Omega)$ such that
\[
\Vert \vn_h|_{\partial\Omega} - \vg\Vert_{L^2(\partial \Omega;\mathbb{R}^3)}\leq C
\Vert \vn_h-\vn\Vert^{1/2} \Vert \vn_h-\vn\Vert_{H^1( \Omega;\mathbb{R}^3)}^{1/2} \le Ch^{1/2} \|\vn\|_{H^1} = \eta_h\to0.
\]

{\it 2. Constraint:} We next show that $\Vert I_h[|\vn_h|^2 - 1]\Vert_{L^1}\to0$. We first bound the error by triangle inequality
\begin{equation}\label{eq:triangle-ineq}
\Vert I_h[|\vn_h|^2 - 1]\Vert_{L^1} \leq \Vert |\vn_h|^2 - I_h[|\vn_h|^2]\Vert_{L^1} + \Vert |\vn_h|^2 - 1\Vert_{L^1}
\end{equation}
We use $ |\vn|^2=1$ a.e., the vector identity $ |\va|^2 - |\vb|^2 = | \va - \vb|^2 +2\vb\cdot(\va-\vb)$, and Cauchy-Schwarz inequality to bound the second term of the RHS of \eqref{eq:triangle-ineq}:
$$
\Vert |\vn_h|^2 - 1\Vert_{L^1} \leq \Vert \vn_h - \vn\Vert^2+2\Vert \vn\Vert\Vert \vn_h- \vn\Vert.
$$
The $L^2$ error estimate $\Vert \vn_h - \vn\Vert\lesssim h \Vert \vn\Vert_{H^1}$ then gives the bound 
\begin{equation}\label{eq:second-term}
\Vert |\vn_h|^2 - 1\Vert_{L^1} \lesssim h(\Vert \vn\Vert_{H^1}+\Vert\vn\Vert)\Vert \vn\Vert_{H^1}.
\end{equation}
The bound on the first term of the RHS of \eqref{eq:triangle-ineq} follows arguments from the proof of \cite[Lemma 2.1]{bartels2016projection}. Over an element $T$, we use an  interpolation estimate in $L^1(T)$ and the fact that $\partial_{\alpha\beta}^2|\vn_h|^2= 2\partial_\alpha\vn_h\cdot \partial_\beta\vn_h$ a.e. in $\Omega$ because $\vn_h$ is piecewise linear to obtain
$$
\Vert |\vn_h|^2 - I_h[|\vn_h|^2]\Vert_{L^1(T)}\lesssim h^2 \Vert D^2[|\vn_h|^2]\Vert_{L^1(T;\mathbb{R}^{3\times3})}\lesssim h^2 \Vert \nabla\vn_h\Vert_{L^2(T;\mathbb{R}^{3\times3})}^2.
$$
Summing over elements and using the $H^1$ stability of the Clement interpolant yields
\begin{equation}\label{eq:first-term}
\Vert |\vn_h|^2 - I_h[|\vn_h|^2]\Vert_{L^1(\Omega)}\lesssim h^2 \Vert \nabla\vn_h\Vert_{L^2(\Omega;\mathbb{R}^{3\times3})}^2\lesssim h^2\Vert \nabla\vn\Vert_{L^2(\Omega;\mathbb{R}^{3\times3})}^2.
\end{equation}
Inserting the estimates \eqref{eq:second-term} and \eqref{eq:first-term} into \eqref{eq:triangle-ineq} shows
$$
\Vert I_h[|\vn_h|^2 - 1]\Vert_{L^1} \lesssim h^2\Vert \nabla\vn\Vert^2+ h\left(\Vert \vn\Vert_{H^1}+\Vert\vn\Vert\right)\Vert \vn\Vert_{H^1}.
$$
Hence, for sufficiently small $h$ depending on shape regularity, $\Vert I_h[|\vn_h|^2 - 1]\Vert_{L^1}\leq C h \|\vn\|_{H^1}^2 \leq C(\Omega)h^{1/2} \|\vn\|_{H^1} = \eta_h$, which goes to 0 as $h\to0$.

{\it 3. Energy}: What is left to show is that the energies converge. Clearly, 
\[
\int_\Omega c_0|\nabla \vn_h|^2+c_1(\divrg \vn_h)^2 \to \int_\Omega c_0|\nabla \vn|^2+c_1(\divrg \vn)^2.
\]
Therefore, we need to show the convergence of the energies for the quartic terms. We focus our attention on $\Vert \vn_h\cdot \curl \vn_h\Vert^2$ first. Note that it suffices to prove 
\begin{equation}\label{eq:twist-limit}
\lim_{h\to0} \Vert \vn_h\cdot \curl \vn_h\Vert= \Vert \vn\cdot \curl \vn\Vert
\end{equation}
because $x\mapsto x^2$ is continuous. By triangle inequality, we have
\[
{\big |}\Vert \vn_h\cdot \curl \vn_h\Vert- \Vert \vn\cdot \curl \vn\Vert{\big|} \leq \Vert \vn_h\cdot (\curl \vn_h - \curl \vn)\Vert + \Vert (\vn_h - \vn)\cdot \curl \vn\Vert
\]
The first term goes to zero because $\Vert \curl \vn_h - \curl \vn \Vert\to0$ and a uniform $L^\infty$ bound on $\vn_h$. For the second term, we extract a pointwise convergent subsequence $\vn_{h_k}\to \vn$ such that $\lim_{h_k\to0}\Vert (\vn_{h_k} - \vn)\cdot \curl \vn\Vert=\limsup_{h\to0}\Vert (\vn_h - \vn)\cdot \curl \vn\Vert$. By the uniform $L^\infty$ bound $\Vert\vn_{h_k}\Vert_{L^\infty(\Omega;\mathbb{R}^3)}\leq 1$, we have the pointwise bound $|(\vn_{h_k} - \vn)\cdot \curl \vn|\leq |\vn_{h_k}\cdot \curl \vn|+|\vn\cdot \curl \vn| \leq 2|\curl \vn|\in L^2(\Omega)$. Hence by dominated convergence theorem, $\Vert (\vn_{h_k} - \vn)\cdot \curl \vn\Vert\to0$, and $\limsup_{h\to0}\Vert (\vn_h - \vn)\cdot \curl \vn\Vert = 0$. Thus, $\lim_{h\to0}\Vert (\vn_h - \vn)\cdot \curl \vn\Vert = 0$, and \eqref{eq:twist-limit} is proved. The same arguments apply to $\Vert \vn_h\times \curl \vn_h\Vert^2$ and the proof is complete.
\end{proof}

\begin{remark}
\rm  
Note that the uniform $L^\infty$ bound on $\vn_h$ is important for Step 3 in the proof of Lemma \ref{lem:recovery-sequence}. This is part of the reason for the enforcement of the $L^\infty$ bound in the definition of $\mathcal{A}_{\vg,h,\eta}$
\end{remark}

\begin{remark}
\rm In contrast to the $L^\infty$ bound on $\vn_h$, we only needed to estimate $I_h[|\vn_h|^2-1]$ in $L^1$. This justifies that the definition of $\mathcal{A}_{\vg,h,\eta}$ involves $\Vert I_h[|\vn_h|^2-1]\Vert_{L^1}$. Moreover, the gradient flow of Section \ref{sec:proj-free} provides a bound for $\|\vn_h\|_{L^\infty}$ and estimates for $\Vert I_h[|\vn_h|^2-1]\Vert_{L^1}$.
\end{remark}


The next two results are important for compactness of minimizers as well as a liminf inequality in the $\Gamma$-convergence framework.

\begin{lemma}[equicoercivity]\label{lem:equicoercivity-discrete}
The modified energy $E$ satisfies  
\begin{equation}
\frac{1}{2}c_0\int_\Omega |\nabla \vn_h|^2d\vx\leq E[\vn_h]
\end{equation}
for all $\vn_h\in\mathcal{A}_{\vg,h,\eta}$.
\end{lemma}
\begin{proof}
  The coercivity from Lemma \ref{lem:equicoercive} (properties of $E$) holds for any $\vn\in H^1(\Omega;\mathbb{R}^3)$,
whence it holds for any $\vn_h\in\mathcal{A}_{\vg,h,\eta}$; see Remark \ref{rem:wlsc}.
\end{proof}

\begin{lemma}[weak lower semicontinuity]\label{lem:discrete-wlsc}
If $\vn_{h,\eta}\in \mathcal{A}_{\vg,h,\eta}$ is such that $\vn_{h,\eta}\rightharpoonup \vn^*$ in
    $H^1(\Omega;\mathbb{R}^3)$ as $\eta,h\to 0$, then
\begin{equation}
E[\vn^*] \leq \liminf_{h,\eta\to0} E[\vn_{h,\eta}].
\end{equation}
\end{lemma}
\begin{proof}
This proof follows the proof of lower semicontinuity of \cite[Lemma 1.4]{hardt1986existence}; see Lemma \ref{lem:equicoercive} and Remark \ref{rem:wlsc}.
\end{proof}
Combining Lemmas \ref{lem:recovery-sequence} (recovery sequence),  \ref{lem:equicoercivity-discrete} (equicoercivity), and \ref{lem:discrete-wlsc} (weak lower semicontinuity) leads to the main convergence result.
\begin{theorem}[convergence of minimizers]\label{thm:conv-min}
Let $h\to0$. There exists a sequence $\{\eta_h\}_h$ with $\eta_h\to0$ as $h\to0$ such that the sequence of minimizers $\vn^*_{h,\eta_h}$ of $E$ over the admissible set $\mathcal{A}_{\vg, h,\eta_h}$ admits a subsequence (not relabeled) $\vn^*_{h,\eta_h}$ such that $\vn^*_{h,\eta_h}\rightharpoonup \vn^*$ in $H^1(\Omega;\mathbb{R}^3)$ and $\vn^*\in \mathcal{A}_{\vg}$ is a minimizer of $E$ over $\mathcal{A}_\vg$. Moreover, $E[\vn^\ast_{h,\eta_h}]\to E[\vn^\ast]$ as $h\to0$.
\end{theorem}
\begin{proof}
The set of minimizers of $E$ in $\mathcal{A}_\vg$ is non-empty in
view of Lemma \ref{lem:equicoercive} (properties of $E$). Let $m = \inf_{\vn\in \mathcal{A}_\vg}E[\vn]$, $R^2 = |\Omega| +2mc_0^{-1} $, and $B_R$ be the $H^1$ ball of radius $R$. 
The set of minimizers $\{\vn\in \mathcal{A}_\vg : E[\vn] = m\}\subseteq B_R$ because $|\vn| = 1$ a.e. implies $\|\vn\|^2\le |\Omega|$ and the estimate \eqref{eq:equicoercivity-estimate} controls $\|\nabla\vn\|^2$.
We now proceed in 3 steps.

\smallskip
{\it 1. Convergence:} Let $\vn\in \mathcal{A}_\vg \cap B_R$. By Lemma \ref{lem:recovery-sequence} (recovery sequence), we have that there is a sequence $\eta = \eta_h = C(\Omega)R h^{1/2}$ such that $\eta_h\to0$ as $h\to0$, and there is a sequence $\vn_{h,\eta}\in \mathcal{A}_{\vg,h,\eta}$ such that $\vn_{h,\eta}\to \vn$ in $H^1(\Omega;\mathbb{R}^3)$ and $\limsup_{h\to0} E[\vn_{h,\eta}] \leq E[\vn]$. 

Using the fact that $\vn^*_{h,\eta}$ is a minimizer of $E$, we have that $E[\vn^*_{h,\eta}]\leq E[\vn_{h,\eta}] $, and
$$
\limsup_{h,\eta\to0} E[\vn^*_{h,\eta}]\leq \limsup_{h,\eta\to0} E[\vn_{h,\eta}]\leq E[\vn].
$$
Thus, $E[\vn^*_{h,\eta}]$ is bounded, and by Lemma \ref{lem:equicoercivity-discrete} (equicoercivity) and the uniform bound $\Vert \vn^*_{h,\eta}\Vert \leq |\Omega|^{1/2}\Vert \vn^*_{h,\eta}\Vert_{L^\infty} \leq C|\Omega|^{1/2}$, we have that there exists a $\vn^*\in H^1(\Omega;\mathbb{R}^3)$ such that there is a subsequence (not relabled) $\vn^*_{h,\eta}\rightharpoonup \vn^*$ in $H^1(\Omega;\mathbb{R}^3)$ as $h\to0$. To see that $\vn^*\in\mathcal{A}_{\vg}$, we need to prove that $\vn^*$ satisfies the unit length constraint pointwise a.e. and the Dirichlet boundary condition. Since $\Vert I_h[|\vn^*_{h,\eta}|^2-1]\Vert_{L^1}\le\eta\to0$, Lemma \ref{lem:unit-length-bartels} (discrete unit length constraint) yields $|\vn^*|=1$ a.e. in $\Omega$.

 We now must show that $\vn^* |_\Omega = \vg$ in the sense of trace. Since the trace operator is weakly continuous from $H^1(\Omega)$ to $L^2(\partial\Omega;\mathbb{R}^3)$, we deduce $(\vn^*_{h,\eta}-\vn^*)|_{\partial\Omega}\rightharpoonup 0$. But $\|\vn^*_{h,\eta}-\vg\|_{L^2(\partial\Omega;\mathbb{R}^3)}\le\eta\to0$, whence $\vn^*|_{\partial\Omega}=\vg$ as desired.

\smallskip
{\it 2. Characterization of $\vn^*$:} We shall now proceed to show that $\vn^*$ is a minimizer. By Lemma \ref{lem:discrete-wlsc} (weak lower semicontinuity), $\liminf_{h,\eta\to0} E[\vn^*_{h,\eta}]\geq E[\vn^*]$. 
We then have 
$$
E[\vn^{*}]\leq \liminf_{h,\eta\to0} E[\vn^*_{h,\eta}]\leq \limsup_{h,\eta\to0} E[\vn^*_{h,\eta}]\leq\limsup_{h,\eta\to0} E[\vn_{h,\eta}]\leq  E[\vn].
$$
Note that $E[\vn^{*}]\leq E[\vn]$ for all $\vn\in \mathcal{A}_{\vg}\cap B_R$, so $\vn^{*}$ is a minimizer due to the choice of $R$.

\smallskip
{\it 3. Energy}: The final claim is $ \lim_{h,\eta\to0} E[\vn^*_{h,\eta}] = E[\vn^{*}]$. Since $E[\vn^{*}]\leq \liminf_{h,\eta\to0} E[\vn^*_{h,\eta}]$, it suffices to prove $\limsup_{h,\eta\to0} E[\vn^*_{h,\eta}]\leq E[\vn^{*}]$. By Lemma \ref{lem:recovery-sequence} (recovery sequence), we construct $\vn_{h,\eta}'\to \vn^*$ in $H^1(\Omega;\mathbb{R}^3)$ such that $\limsup_{h,\eta\to0} E[\vn'_{h,\eta}]\leq  E[\vn^*]$. We use the assumption that $\vn^*_{h,\eta}$ is a minimizer of $E$, i.e. $E[\vn_{h,\eta}^*] \le E[\vn_{h,\eta}']$, to prove
$$
\limsup_{h,\eta\to0} E[\vn^*_{h,\eta}]\leq\limsup_{h,\eta\to0} E[\vn'_{h,\eta}]\leq E[\vn^{*}],
$$
which is the desired bound. This completes the proof.
\end{proof}
\begin{remark}\label{rmk:regularity}
  \rm
The present theory only requires $H^1(\Omega;\mathbb{R}^3)$-regularity of the solution, and thus allows for point defects of LC. This contrasts with the theory in other papers \cite{adler2015energy,bartels2022quasi, hu2009saddle}, which require higher regularity. The higher regularity requirements in \cite{bartels2022quasi, hu2009saddle} yield error estimates for harmonic maps while the higher regularity in \cite{adler2015energy} provides error estimates for solving the Newton linearizations of the Frank-Oseen energy. On the other hand, the $\Gamma$-convergence theory does not provide error estimates. They require a different approach.
\end{remark}

\section{Projection-free gradient flow for discrete problem}\label{sec:proj-free}

In this section, we propose a gradient flow algorithm to compute critical points of $E$ over the discrete admissible set $\mathcal{A}_{\vg,h,\eta}$. The main idea follows that of \cite{bartels2016projection,bartels2022stable} to gain control of the violation of the unit length constraint and quartic nonlinearity. Recall the modified full Frank energy:
\begin{equation*}
E[\vn_h] = E_1[\vn_h] + E_2[\vn_h],
\end{equation*}
where $E_1$ is the quadratic part of the energy and $E_2$ contains the quartic contributions
\begin{align*}
E_1[\vn_h] &:= \frac{1}{2}\int_\Omega c_0|\nabla \vn|^2 + c_1(\divrg \vn)^2d\vx,\\
 E_2[\vn_h] &:= \frac{1}{2}\int_\Omega c_2(\vn\cdot {\curl} \vn)^2 +c_3|\vn\times {\curl} \vn|^2 d\vx.
\end{align*}

The gradient flow involves a minimization problem at each step. Our goal is to find an increment $d_t\vn^{k+1}_h$, and set $\vn_h^{k+1} := \vn_h^{k}+\tau d_t\vn^{k+1}_h$. In order to make sure the minimization problem involves a linear problem to solve, there are two linearizations to consider.

We first linearize the constraint. Rather than enforcing $|\vn^{k+1}_h(z)|=1$, which is a nonconvex constraint,
  we search for $d_t\vn^{k+1}_h\in \mathbb{T}(\vn_h^{k})$ in the tangent space $\mathbb{T}(\vn_h^{k})$ to $\vn_h^{k}$. Figure \ref{fig:bartels-idea} shows what an increment $d_t\vn^{k+1}_h$ looks like at a node $z$. Moreover, searching in tangent directions within $\mathbb{T}(\vn_h^{k})$ allows for control of the constraint violation in terms of $\tau$
\[
\big|\vn^{k}_h(z)+\tau d_t\vn^{k+1}_h(z)\big|^2 = \big|\vn^{k}_h(z)\big|^2 + \tau^2 \big|d_t\vn^{k+1}_h(z)\big|^2.
\]
The second linearization acts on $E_2$, and entails the minimization problem for $d_t\vn^{k+1}_h$
\begin{equation}
\tau d_t\vn^{k+1}_h\in \text{argmin}_{\vv_h\in \mathbb{T}_h(\vn^{k}_h)}\left(\frac{1}{2\tau}\Vert \vv_h\Vert^2_{*}+ E_1[\vn^{k}_h+\vv_h ]+\frac{\delta E_2[\vn^{k}_h; \vv_h]}{\delta\vn}\right),
\end{equation}
where $\Vert \cdot \Vert_{*}$ is a norm induced by some flow metric. In order for $\frac{\delta E_2[\vn^{k}_h; \tau d_t\vn^{k+1}_h]}{\delta\vn}$ to be a bounded and controlled quantity, we ought to control $d_t\vn^{k+1}_h$ in $H^1$ , and the linearization of $E_2$ ought to be continuous on $H^1$, which means $\vn^{k}_h$ needs to be bounded in $L^\infty$. The desired control of $d_t\vn^{k+1}_h$ in $H^1$ motivates the choice of the $H^1$ norm for the flow metric i.e.\ $\Vert \vv_h\Vert^2_{*} =\Vert \nabla\vv_h\Vert^2$. Control of $\Vert \nabla d_t\vn^{k+1}_h\Vert^2$ and the inverse inequality from $L^\infty$ to $H^1$ in Lemma \ref{lem:discrete-sob} (discrete Sobolev inequality), dictates a stability constraint between $h$ and $\tau$. The resulting linear system for $ d_t\vn^{k+1}_h,$ reads
\begin{equation}\label{eq:grad-flow-eq}
(1+c_0\tau) (\nabla d_t\vn^{k+1}_h,\nabla \vv_h)+c_1\tau(\divrg d_t\vn^{k+1}_h,\divrg \vv_h) = -\frac{\delta E[\vn^{k}_h;\vv_h]}{\delta\vn} \quad \forall \vv_h\in \mathbb{T}(\vn_h^k).
\end{equation}
To recap, there are three main ingredients:
\begin{itemize}
\item Control of $\Vert \nabla d_t\vn^{k+1}_h\Vert^2$ from the flow metric. 
\item Control the violation of the unit length constraint in terms of the pseudotime-step $\tau$.
\item Control of the linearization of $E_2$ upon combining uniform bounds for $\|\vn_h^k\|_{L^\infty}$
  and $\Vert \nabla d_t\vn^{k+1}_h\Vert^2$.
\end{itemize}
This strategy originated in the context of bilayer plates \cite{bartels2022stable} and was also used in \cite{bonito2023gamma}.
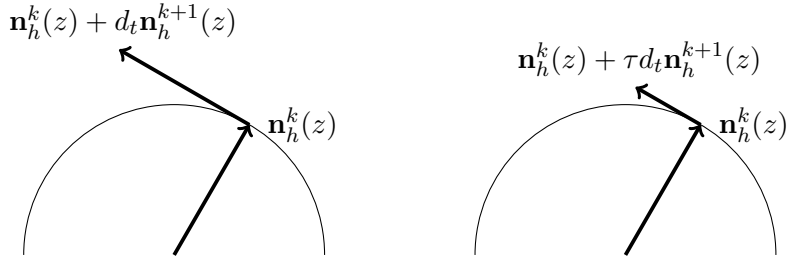
\begin{figure}[htbp]
\begin{center}
\begin{tikzpicture}
  \coordinate (A) at (0,0);
  \coordinate (B) at (8,0);
  \coordinate (C) at ($(A)!.5!(B)$);
  \coordinate (D) at ($(A)!.5!(C)$);
  \coordinate (E) at ($(B)!.5!(C)$);
  \coordinate[label=right:$\;\vn_h^k(z)$] (F) at (3,1.732);
  \coordinate[label=right:$\;\vn_h^k(z)$] (G) at (9,1.732);
  \coordinate[label=above:$\;\vn_h^k(z)+d_t\vn^{k+1}_h(z)$] (H) at (1.268,2.732);
  \coordinate[label=above:$\;\vn_h^k(z)+\tau d_t\vn^{k+1}_h(z)$] (I) at (8.134,2.232);
  \draw[->, line width=0.5mm] (2,0)-- (F);
  \draw[->, line width=0.5mm] (8,0)-- (G);
  \draw[->, line width=0.5mm] (F)--(H);
  \draw[->, line width=0.5mm] (G)--(I);
  \draw (A) let \p1 = ($(C) - (D)$), \n1={veclen(\x1,\y1)} in arc[start angle=180, end angle=0, radius=\n1];
  \draw (E) let \p1 = ($(C) - (E)$), \n1={veclen(\x1,\y1)} in arc[start angle=180, end angle=0, radius=\n1];
\end{tikzpicture}
\end{center}
\caption[Projection-free gradient flow idea]{By searching in tangent directions and damping with parameter $\tau$ yields $|\vn^{k}_h(z)+\tau d_t\vn^{k+1}_h(z)|=|\vn^k_h(z)|^2+\tau^2|d_t\vn^{k+1}_h(z)|^2$.}\label{fig:bartels-idea}
\end{figure}

The resulting gradient flow algorithm is below.

\RestyleAlgo{boxruled}
\begin{algorithm}[H]
\caption{Projection-free gradient flow}\label{alg:grad-flow}
\KwData{Triangulation $\Th$ with meshsize $h$, pseudotime-step $\tau$, stopping tolerance $\veps$, and initial guess $\vn^{0}_h\in \mathbb{V}_h$}
\KwResult{Approximate discrete local minimizer $\vn^*_{h,\tau,\veps}$}
$k\gets0$\\
\While{$E^{k-1}-E^{k}\geq \tau\veps$}{
Compute increment $d_t\vn^{k+1}_h\in \mathbb{T}(\vn^{k}_h)$ to solve \eqref{eq:grad-flow-eq}\\
Update: $\vn^{k+1}_h = \vn^{k}_h+\tau d_t\vn^{k+1}_h$
}
\end{algorithm}

\smallskip
We see the following property of Algorithm \ref{alg:grad-flow} immediately from Figure \ref{fig:bartels-idea}.
\begin{remark}[lower bound on $|\vn_h^{k}(z)|^2$]\label{rmk:lower-bound-length}
\rm
Given $z\in \mathcal{N}_h$, we always have $|\vn_h^{k}(z)|^2\geq1$ if $|\vn_h^{0}(z)|^2=1$. This is because $d_t\vn_h^{k}\in \mathbb{T}(\vn^{k}_h)$, and
\[
|\vn_h^{k+1}(z)|^2 = |\vn_h^{k}(z)|^2 +\tau^2|d_t\vn_h^{k}(z)|^2\geq |\vn_h^{k}(z)|^2
\]
Applying an induction argument yields $|\vn_h^{k}(z)|^2\geq |\vn_h^{0}(z)|^2\geq1$.
\end{remark}

\subsection{Properties of the gradient flow}

Algorithm \ref{alg:grad-flow} (projection-free gradient flow) has a few desirable properties. The most important property is the following energy stability.

\begin{theorem}[energy stability and $L^\infty$ control of constraint]\label{thm:energy_decrease}
  Let $\vn^{0}_h\in \mathbb{V}_h$ be such that $|\vn^{0}_h(z)|^2=1$ for all $z\in \mathcal{N}_h$. There is a constant $0<\Cstab\leq1$ which may depend on $E[\vn^{0}_h], c_{inv},$ and $c_i$ for $i=0,1,2,3$ such that if $\tau h^{-1}\leq \Cstab$ then, for all $k$
\begin{equation}\label{eq:energy-ineq}
E[\vn^{k+1}_h] +\frac{\tau}{2} \sum_{\ell=1}^k\Vert \nabla d_t\vn^{\ell+1}_h\Vert^2 \leq E[\vn^{0}_h],
\end{equation}
and for all $z\in \mathcal{N}_h$
\begin{equation}\label{eq:pointwise-ineq}
  \left||\vn^{k+1}_h(z)|^2-1\right| \leq 4c_{inv}^2 \tau h^{-1} E[\vn^{0}_h],
\end{equation}
where $c_{inv}$ is the constant from Lemma \ref{lem:discrete-sob} (discrete Sobolev inequality).
\end{theorem}
\begin{proof}
For simplicity of presentation, we distinguish two cases depending on whether $c_1=c_3=0$ or not. We split the proof in several steps.

1. {\it Induction hypothesis}: We assume for $k\geq0$ that
\begin{gather}
E[\vn^{k}_h] +\frac{\tau}{2} \Vert \nabla d_t\vn^{k}_h\Vert^2 \leq E[\vn^{k-1}_h] \label{eq:inductive-hyp-1},\\
0\leq |\vn^{k}_h(z)|^2-1 \leq 4c_{inv}^2\tau h^{-1}E[\vn_h^{0}] \label{eq:inductive-hyp-2},
\end{gather}
with $\vn^{-1}_h = \vn^{0}_h$; this is trivially satisfied for $k=0$. Letting $c_1=c_3=0$ and testing \eqref{eq:grad-flow-eq} with $\tau d_t\vn^{k+1}_h$ yields
\begin{align*}
(\tau+c_0\tau^2)\Vert \nabla d_t\vn^{k+1}_h\Vert^2  = &-c_0\tau(\nabla\vn^{k}_h,\nabla d_t\vn^{k+1}_h) \\
&-c_2\tau(\vn^{k}_h\cdot\curl \vn^{k}_h, \vn^{k}_h\cdot\curl d_t\vn^{k+1}_h +d_t\vn^{k+1}_h\cdot\curl \vn^{k}_h).
\end{align*}
We next exploit this relation to show that \eqref{eq:inductive-hyp-1} and \eqref{eq:inductive-hyp-2} hold for $k+1$. It is clear that adding \eqref{eq:inductive-hyp-1} over $k$, telescopic cancellation yields the desired energy estimate.

\medskip
2. {\it Bound on $\Vert\nabla\vn_h^{k+1}\Vert$}: Recall that $\tau d_t\vn^{k+1}_h= \vn^{k+1}_h - \vn^{k}_h$. By using the equality $(\mathbf{b},\mathbf{b-a}) = \frac{1}{2}(\Vert \mathbf{b}\Vert^2 -\Vert \mathbf{a}\Vert^2 +\Vert \mathbf{b-a}\Vert^2)$, we have
\begin{equation*}
-c_0\tau(\nabla\vn^{k}_h,\nabla d_t\vn^{k+1}_h) = \frac{c_0}{2} \Vert \nabla\vn^{k}_h\Vert^2 -\frac{c_0}{2} \Vert \nabla\vn^{k+1}_h\Vert^2 +\frac{c_0\tau^2}{2} \Vert \nabla d_t\vn^{k+1}_h\Vert^2.
\end{equation*}
Inserting this into the original equation and rearranging, we have
\begin{equation}\label{eq:quadratic-energy}
\Big(\tau+c_0\frac{\tau^2}{2}\Big)\Vert \nabla d_t\vn^{k+1}_h\Vert^2+\frac{c_0}{2} \Vert \nabla\vn^{k+1}_h\Vert^2 = \frac{c_0}{2} \Vert \nabla\vn^{k}_h\Vert^2+I, 
\end{equation}
where
\[
I = -c_2\tau \big(\vn^{k}_h\cdot\curl \vn^{k}_h, \vn^{k}_h\cdot\curl d_t\vn^{k+1}_h +d_t\vn^{k+1}_h\cdot\curl \vn^{k}_h\big) .
\]
Applying Cauchy-Schwarz and triangle inequalities to $I$ yields
\begin{align*}
  |I| &\leq \tau c_2\Vert\vn^{k}_h\cdot\curl \vn^{k}_h\Vert\; \left(\Vert\vn^{k}_h\cdot\curl d_t\vn^{k+1}_h\Vert +\Vert d_t\vn^{k+1}_h\cdot\curl \vn^{k}_h\Vert\right)
  \\ &
  \le \tau c_2\Vert\vn^{k}_h\cdot\curl \vn^{k}_h\Vert\;
  \left(\Vert\vn^{k}_h\Vert_{L^\infty}\Vert\nabla d_t\vn^{k+1}_h\Vert +\Vert d_t\vn^{k+1}_h\Vert_{L^\infty} \Vert\nabla \vn^{k}_h\Vert\right).
\end{align*}
Note that $c_2\Vert\vn^{k}_h\cdot\curl \vn^{k}_h\Vert^2\leq 2E[\vn^{0}_h]$ by the inductive hypothesis \eqref{eq:inductive-hyp-1}. Hence,
\begin{align*}
|I| \leq\tau\sqrt{2c_2E[\vn^{0}_h]}\left(\Vert\vn^{k}_h\Vert_{L^\infty}\Vert\nabla d_t\vn^{k+1}_h\Vert +\Vert d_t\vn^{k+1}_h\Vert_{L^\infty} \Vert\nabla \vn^{k}_h\Vert\right).
\end{align*}
The induction hypotheses \eqref{eq:inductive-hyp-1} and \eqref{eq:inductive-hyp-2} imply $\Vert\vn^{k}_h\Vert_{L^\infty} \leq 1+4c_{inv}^2\tau h^{-1}E[\vn^{0}_h]\leq 1+4c_{inv}^2E[\vn^{0}_h]$ as well as $\Vert\nabla \vn^{k}_h\Vert \leq \sqrt{\frac{2}{c_0}E[\vn^{0}_h]}$. Incorporating these expressions into the above estimate and utilizing Lemma \ref{lem:discrete-sob} (discrete Sobolev inequality) yields
\begin{align*}
|I| &\leq \tau\sqrt{2c_2E[\vn^{0}_h]}\left(\big(1+4c_{inv}^2E[\vn^{0}_h]\big)+c_{inv}h^{-1/2} \sqrt{\frac{2}{c_0}E[\vn^{0}_h]} \right)\Vert\nabla d_t\vn^{k+1}_h\Vert\\
&\leq c' \tau h^{-1/2}\Vert\nabla d_t\vn^{k+1}_h\Vert, 
\end{align*}
where $c'$ depends on $E[\vn^{0}_h], c_0, c_2$ and $c_{inv}$. We then apply Young's inequality to further estimate
\[
|I| \leq \frac{{c'}^2}{2}+\frac{\tau^2h^{-1}}{2}\Vert\nabla d_t\vn^{k+1}_h\Vert^2.
\]
Since $\tau h^{-1}\leq 1$, we absorb the last term into the left hand side of \eqref{eq:quadratic-energy} and obtain again using the inductive hypothesis \eqref{eq:inductive-hyp-1}
\begin{equation}\label{eq:int-est-grad}
\frac{\tau}{2}\Vert \nabla d_t\vn^{k+1}_h\Vert^2+\frac{c_0}{2} \Vert \nabla\vn^{k+1}_h\Vert^2\leq\frac{c_0}{2} \Vert \nabla\vn^{k}_h\Vert^2+ \frac{{c'}^2}{2} \leq E[\vn^{0}_h]+ \frac{{c'}^2}{2} \leq c'',
\end{equation}
where $c''$ only depends on $E[\vn^{0}_h]$ and $c'$, and so only depends on $E[\vn^{0}_h], c_0, c_2$ and $c_{inv}$.

\medskip
{\it 3. Intermediate estimate for $\Vert\vn^{k+1}_h\Vert_{L^\infty}$}: Property $d_t\vn^{k+1}_h\in \mathbb{T}_h(\vn^{k}_h)$ implies $d_t\vn^{k+1}_h(z)\cdot\vn^{k}_h(z) = 0$ at nodes $z\in\mathcal{N}_h$, whence 
\begin{equation*}
| \vn^{k+1}_h(z)|^2 = | \vn^{k}_h(z)+\tau d_t\vn^{k+1}_h(z)|^2 = | \vn^{k}_h(z)|^2+\tau^2|d_t\vn^{k+1}_h(z)|^2.
\end{equation*}
By the inductive hypothesis \eqref{eq:inductive-hyp-2} and the assumption $\tau h^{-1}\leq 1$, we deduce $\Vert\vn^{k}_h\Vert_{L^\infty}^2 \leq 1+4c_{inv}^2E[\vn^{0}_h]$ and
\[
| \vn^{k+1}_h(z)|^2 \leq  \left(1+4c_{inv}^2E[\vn^{0}_h]\right)^2+\tau^2\Vert d_t\vn^{k+1}_h\Vert_{L^\infty}^2.
\]
Applying again Lemma \ref{lem:discrete-sob} (discrete Sobolev inequality) and the assumption $\tau h^{-1}\leq 1$, we now deduce an estimate on $| \vn^{k+1}_h(z)|^2$, namely
\begin{equation}\label{eq:int-est}
\begin{aligned}
  | \vn^{k+1}_h(z)|^2 &\leq \left(1+4c_{inv}^2E[\vn^{0}_h]\right)^2+ \tau^2h^{-1}c_{inv}^2\Vert\nabla d_t\vn^{k+1}_h\Vert^2
  \\
& \leq \left(1+4c_{inv}^2E[\vn^{0}_h]\right)^2 + 2\tau h^{-1}c_{inv}^2c'' \leq c''', 
\end{aligned}
\end{equation}
where $c'''$ only depends on $E[\vn^{0}_h], c_{inv},$ and $c''$. This is the desired intermediate estimate for $\Vert \vn^{k+1}_h\Vert_{L^\infty}^2$ but is not quite \eqref{eq:inductive-hyp-2} for $k+1$.

\medskip
{\it 4. Energy estimate}: To prove the asserted energy estimate \eqref{eq:inductive-hyp-1}, we rewrite $I$ in \eqref{eq:quadratic-energy}. To this end, let
\[
\va_k = \vn^{k}_h,\quad \vb_k = \curl \vn^{k}_h
\]
and note that
\begin{align*}
\va_{k+1}\cdot\vb_{k+1} - \va_{k}\cdot\vb_{k} =& \, (\va_{k+1} - \va_{k})\cdot\vb_k \\
&+ \va_k\cdot(\vb_{k+1} - \vb_{k})\\
&+(\va_{k+1} - \va_{k})\cdot(\vb_{k+1} - \vb_{k}).
\end{align*}
Squaring and rearranging terms, we end up with
\begin{align*}
|\va_{k+1}\cdot\vb_{k+1}|^2  =& \, |\va_{k}\cdot\vb_{k}|^2 \\
&+ 2\va_{k}\cdot\vb_{k}\left[(\va_{k+1} - \va_{k})\cdot\vb_k+ \va_k\cdot(\vb_{k+1} - \vb_{k})\right]\\
&+ 2\va_{k}\cdot\vb_{k}(\va_{k+1} - \va_{k})\cdot(\vb_{k+1} - \vb_{k})\\
&+ \left|(\va_{k+1} - \va_{k})\cdot\vb_k+\va_{k+1} \cdot(\vb_{k+1} - \vb_{k})\right|^2.
\end{align*}
In view of the definition of $\va_k$ and $\vb_k$, multiplying by $\frac{c_2}{2}$ and integrating over $\Omega$
yields
\begin{equation}\label{eq:twist-energy}
\begin{aligned}
\frac{c_2}{2}\norm{\vn_h^{k+1}\cdot \curl \vn_h^{k+1} }^2 = & \, \frac{c_2}{2}\norm{\vn_h^{k}\cdot \curl \vn_h^{k}}^2 \\
&+c_2\tau (\vn_h^{k}\cdot \curl \vn_h^{k}, d_t\vn_h^{k}\cdot \curl \vn_h^{k}+\vn_h^{k}\cdot \curl d_t\vn_h^{k+1}) \\
&+ c_2\tau^2(\vn_h^{k}\cdot \curl \vn_h^{k}, d_t\vn_h^{k+1}\cdot \curl d_t\vn_h^{k+1}) \\
&+\frac{c_2\tau^2}{2}\norm{ d_t\vn_h^{k+1}\cdot \curl \vn_h^{k}+\vn_h^{k+1}\cdot \curl d_t\vn_h^{k+1}}^2 .
\end{aligned}
\end{equation}
Adding \eqref{eq:quadratic-energy} and \eqref{eq:twist-energy} and canceling the order $\tau$ term with $I$, we obtain
\begin{equation} \label{eq:int-est-energy}
\Big(\tau+c_0\frac{\tau^2}{2}\Big)\Vert \nabla d_t\vn^{k+1}_h\Vert^2+ E[\vn_h^{k+1}]
= E[\vn_h^{k}]+ II +III,
\end{equation}
where
\begin{align*}
II&= c_2\tau^2(\vn_h^{k}\cdot \curl \vn_h^{k}, d_t\vn_h^{k+1}\cdot \curl d_t\vn_h^{k+1}) \nonumber\\
III&=\frac{c_2\tau^2}{2}\norm{ d_t\vn_h^{k+1}\cdot \curl \vn_h^{k}+\vn_h^{k+1}\cdot \curl d_t\vn_h^{k+1}}^2.
\end{align*}
To derive the energy inequality, we will estimate $II$ and $III$ separately. We first estimate $II$ by Cauchy-Schwarz and the inductive hypothesis \eqref{eq:inductive-hyp-1}:
\[
|II| \leq c_2\tau^2\norm{\vn_h^{k}\cdot \curl \vn_h^{k}}\norm{d_t\vn_h^{k+1}\cdot \curl d_t\vn_h^{k+1}}\leq \tau^2\sqrt{2c_2E[\vn_h^{k}]}\norm{d_t\vn_h^{k+1}}_{L^\infty} \norm{\curl d_t\vn_h^{k+1}}.
\]
We then apply H\"{o}lder inequality and Lemma \ref{lem:discrete-sob} (discrete Sobolev inequality) to estimate $\Vert d_t\vn_h^{k+1}\Vert_{L^\infty} \leq c_{inv} h^{-1/2}\Vert\nabla d_t\vn_h^{k+1}\Vert$ whence
\begin{align*}
|II| \leq
\sqrt{2c_2E[\vn^{0}_h]}c_{inv}\tau^2h^{-1/2}\norm{\nabla d_t\vn_h^{k+1}}^2 \leq c^{iv}\tau^2h^{-1}\norm{\nabla d_t\vn_h^{k+1}}^2,
\end{align*}
where $c^{iv}$ depends on $c_2, E[\vn^{0}_h],$  and $c_{inv}$.
Moreover, we estimate $III$ as follows using the inequality $|\va+\vb|^2\leq 2|\va|^2+2|\vb|^2$ and H\"{o}lder inequality:
\begin{align*}
|III| &\leq c_2\tau^2\left(\norm{ d_t\vn_h^{k+1}\cdot \curl \vn_h^{k}}^2+\norm{\vn_h^{k+1}\cdot \curl d_t\vn_h^{k+1}}^2\right)\\
&\leq c_2\tau^2\left(\norm{ d_t\vn_h^{k+1}}_{L^\infty}^2\norm{ \nabla \vn_h^{k}}^2+\norm{\vn_h^{k+1}}_{L^\infty}^2\norm{ \nabla d_t\vn_h^{k+1}}^2\right).
\end{align*}
Combining the energy decrease from the inductive hypothesis \eqref{eq:inductive-hyp-1}, the intermediate estimate on $\Vert \vn_h^{k+1}\Vert^2_{L^\infty}$ from \eqref{eq:int-est}, and Lemma \ref{lem:discrete-sob} (discrete Sobolev inequality) helps us further bound $III$ as follows:
\begin{align*}
|III|&\leq c_2\tau^2\left(\norm{ d_t\vn_h^{k+1}}_{L^\infty}^2\frac{2}{c_0}E[\vn^{0}_h] + c''' \norm{ \nabla d_t\vn_h^{k+1}}^2\right)\\
&\leq c_2\tau^2\left(c_{inv}^2h^{-1}\frac{2}{c_0}E[\vn^{0}_h] + c''' \right)\norm{ \nabla d_t\vn_h^{k+1}}^2
\leq c^{v} \tau^2 h^{-1} \norm{ \nabla d_t\vn_h^{k+1}}^2,
\end{align*}
where $c^{v}$ depends on $c_2,c_{inv}, c'''$ and $E[\vn^{0}_h]$.
Inserting the estimates of $II$ and $III$ into \eqref{eq:int-est-energy} yields the following inequality
\begin{align*}
\left(\tau+c_0\frac{\tau^2}{2}\right)\Vert \nabla d_t\vn^{k+1}_h\Vert^2+ E[\vn_h^{k+1}]
\leq E[\vn_h^{k}] + c^{vi} \tau^2 h^{-1} \norm{ \nabla d_t\vn_h^{k+1}}^2,
\end{align*}
where $c^{vi}$ depends only on $c_{inv}, E[\vn^{0}_h],c_0,c_2$. We now pick $\tau$ so that $\tau h^{-1}\leq \Cstab :=\min\{1,\frac{1}{2c^{vi}}\}$ to obtain the desired energy inequality \eqref{eq:inductive-hyp-1} for $k+1$:
\begin{align*}
\frac{\tau}{2}\Vert \nabla d_t\vn^{k+1}_h\Vert^2+ E[\vn_h^{k+1}]\leq E[\vn_h^{k}].
\end{align*}

\medskip
{\it 5. Constraint for $\vn_h^{k+1}$}: Recalling the orthogonality property
$$|\vn_h^{k+1}(z)|^2 = |\vn_h^{k}(z)|^2+\tau^2|d_t\vn_h^{k+1}(z)|^2,$$
and applying again Lemma \ref{lem:discrete-sob} (discrete Sobolev inequality) gives
\begin{align*}
0 \le  |\vn_h^{k+1}(z)|^2 - |\vn_h^{k}(z)|^2 = \tau^2h^{-1} c_{inv}^2 \norm{\nabla d_t\vn_h^{k+1}}^2
\leq 2c_{inv}^2\tau h^{-1}\left(E[\vn_h^{k}] - E[\vn_h^{k+1}]\right).
\end{align*}
Since $|\vn_h^{k+1}(z)|^2 \geq |\vn_h^{k}(z)|^2\geq 1$,
summing over $k$ and using telescoping cancellation yields
$$
0 \le |\vn_h^{k+1}(z)|^2 - 1 \leq 2c_{inv}^2\tau h^{-1}\left(E[\vn_h^{0}] - E[\vn_h^{k+1}]\right) \leq 2c_{inv}^2\tau h^{-1}E[\vn_h^{0}],
$$
because $|\vn^{0}_h(z)| = 1$ and $E[\vn_h^{k+1}]\geq0$. These two inequalities are the desired nodal length violation in \eqref{eq:inductive-hyp-2} for $k+1$, and complete the inductive argument provided $c_1=c_3=0$.

\medskip
{\it 6. Case $c_1\ne0, c_3\ne0$}: We now verify \eqref{eq:inductive-hyp-1} and \eqref{eq:inductive-hyp-2},
  first for $c_1\ne0$ and next for $c_3\ne0$. Since the splay term $c_1\ne0$ is dealt with implicitly, we immediately get the energy decrease of this term using similar quadratic identities.
If $c_3\neq0$, instead, there are three steps that need to be checked. First, one would need to ensure that the intermediate estimates in \eqref{eq:int-est-grad} and \eqref{eq:int-est} remains valid. This is indeed the case because an application of the preceding techniques shows there is a larger constant $c'$ such that
\[ c_3\tau \left|\left(\vn^{k}_h\times\curl \vn^{k}_h, \vn^{k}_h\times\curl d_t\vn^{k+1}_h +d_t\vn^{k+1}_h\times\curl \vn^{k}_h\right)\right| \leq c'\tau h^{-1/2}\Vert \nabla d_t\vn_h^{k+1}\Vert,
\]
and using Young's inequality yields an estimate similar to \eqref{eq:int-est-grad}. Then remaining intermediate estimate \eqref{eq:int-est} readily follows.

The next key step would be to achieve a version of \eqref{eq:twist-energy}. Since the quartic structure of the bend term is similar to the twist term in \eqref{eq:twist-energy}, with dot products replaced by cross products, the desired energy inequality emerges from the same arguments developed to estimate the remainder terms in \eqref{eq:int-est-energy}, possibly with a smaller constant $\Cstab$.
\end{proof}

An interesting observation is that once energy stability is achieved, one does not need to take $\tau h^{-1}\to 0$ to recover control of the unit length constraint violation. In fact, if we measure the constraint violation in a weaker norm, then taking $\tau\to0$ would recover the unit length constraint as long as $\tau h^{-1}\leq \Cstab$ where $\Cstab$ is the constant from Theorem \ref{thm:energy_decrease} (energy stability and $L^\infty$ control of constraint). We explore this next.
\begin{cor}[control of $L^1$ violation of constraint]\label{cor:L1-constraint}
Let $\vn^{0}_h\in \mathbb{V}_h$ such that $|\vn^{0}_h(z)|^2=1$ for all $z\in \mathcal{N}_h$. Suppose $\tau h^{-1}\leq \Cstab$, where $\Cstab$ is the constant from Theorem \ref{thm:energy_decrease} (energy stability and $L^\infty$ control of constraint). Then
$$
\norm{I_h \big[|\vn_h^{k+1}|^2-1 \big]}_{L^1}\lesssim \tau E[\vn_h^{0}].
$$
\end{cor}
\begin{proof}
Suppose $\tau h^{-1}\leq \Cstab$. In view of the nodal orthogonality property
$$
  |\vn_h^{k+1}(z)|^2 = |\vn_h^{k}(z)|^2+\tau^2|d_t\vn_h^{k+1}(z)|^2,
$$
adding over $k$ and using telescopic cancellation along with $|\vn_h^0(z)|=1$ yields
\[
|\vn_h^{k+1}(z)|^2 - 1 = \tau^2 \sum_{\ell=0}^k|d_t\vn_h^{\ell+1}(z)|^2.
\]
Multiplying by the measure $|\omega_z|$ of the star $\omega_z$, and recalling the quadrature identity
$\frac{1}{3} \sum_{z\in\mathcal{N}_h} v_h(z) |\omega_z| = \int_\Omega v_h$ for all $v_h\in\V_h$, leads to
\[
\norm{I_h \big[|\vn_h^{k}|^2 -1  \big]}_{L^1} = \tau^2\sum_{\ell=0}^k \norm{I_h\big[|d_t\vn_h^{\ell+1}|^2\big]}_{L^1}
\lesssim \tau^2\sum_{\ell=0}^k \norm{d_t\vn_h^{\ell+1}}^2
\]
Applying Poincar\'e inequality $\Vert d_t\vn_h^{k+1}\Vert^2\lesssim \Vert \nabla d_t\vn_h^{k+1}\Vert^2$ in
conjunction with \eqref{eq:energy-ineq}, implies
\[
\norm{I_h \big[|\vn_h^{k+1}|^2 -1  \big]}_{L^1}  \lesssim \tau^2\sum_{\ell=0}^k\Vert \nabla d_t\vn_h^{\ell+1}\Vert^2
\lesssim  \tau E[\vn_h^{0}].
\]
This is the asserted estimate.

\end{proof}

The next two results establish that Algorithm \ref{alg:grad-flow} (projection-free gradient flow) computes a critical point of $ E$ in the discrete admissible set $\mathcal{A}_{\vg,h,\eta}$. They mimic results for harmonic maps \cite[Lemma 3.8.]{bartels2005stability},\cite[Proposition 3.1]{bartels2016projection}.
\begin{cor}[residual estimate]\label{cor:small_residual}
Given $\varepsilon>0$, there is an integer $k_\varepsilon$ such that $ E[\vn_h^{k_\varepsilon - 1}] -  E[\vn_h^{k_\varepsilon}] <\varepsilon\tau$. Moreover, $\vn_h^{k_\varepsilon}$ satisfies
\begin{equation}
\left|\frac{\delta E[\vn_h^{k_\varepsilon}; \vv_h]}{\delta \vn}\right|\leq (1+\tau c_0+\tau c_1)\sqrt{2\varepsilon} \Vert \nabla \vv_h\Vert
\end{equation}
for all $\vv_h\in \mathbb{T}(\vn_h^{k_\varepsilon})$.
\end{cor}
\begin{proof}
The fundamental energy  estimate \eqref{eq:inductive-hyp-1} implies
$ E[\vn^{k-1}_h] -  E[\vn^{k}_h]\geq 0$ along with
\begin{equation*}
0\leq \sum_{k=0}^{K}\left( E[\vn^{k-1}_h] -  E[\vn^{k}_h]\right) =  E[\vn^{0}_h] -  E[\vn^{K}_h]\leq  E[\vn^{0}_h].
\end{equation*} 
Therefore, the series $\sum_{k=1}^{\infty} \big( E[\vn^{k-1}_h] -  E[\vn^{k}_h] \big)$  of non-negative
terms converges, and
\begin{equation*}
\lim_{k\to\infty}\left( E[\vn^{k-1}_h] -  E[\vn^{k}_h]\right) = 0.
\end{equation*}
Hence, there exists a $k_\varepsilon$ such that $ E[\vn^{k_\varepsilon}_h] -  E[\vn^{k_\varepsilon+1}_h]\leq \varepsilon\tau$.
Moreover, \eqref{eq:inductive-hyp-1} yields
\begin{equation*}
  \norm{\nabla d_t \vn^{k_\varepsilon+1}_h}^2\leq \frac{2}{\tau}\left( E[\vn^{k_\varepsilon}_h] -  E[\vn^{k_\varepsilon+1}_h]\right)\leq 2\varepsilon.
\end{equation*}
Using the gradient flow equation in \eqref{eq:grad-flow-eq}, we realize that
\begin{equation*}
\frac{\delta  E[\vn^{k_\varepsilon}_h,\vv_h]}{\delta\vn} = (1+\tau c_0)(\nabla d_t\vn^{k_\varepsilon+1}_h, \nabla \vv_h) +\tau c_1(\divrg d_t\vn^{k_\varepsilon+1}_h, \divrg \vv_h),
\end{equation*}
whence the asserted estimate
\begin{equation*}
\left|\frac{\delta  E[\vn^{k_\varepsilon}_h,\vv_h]}{\delta\vn}\right| \leq (1+\tau c_0+\tau c_1)\norm{ \nabla d_t\vn^{k_\varepsilon+1}_h}\norm{\nabla \vv_h}\leq (1+\tau c_0+\tau c_1)\sqrt{2\varepsilon}\norm{\nabla \vv_h}
\end{equation*}
follows immediately.
\end{proof}

A serious difficulty to prove convergence of $\{\vn_h^{h_\eps}\}_{\eps>0}$ is the fact that the tangent space
  $\mathbb{T}(\vn_h^{k_\eps})$ depends on $\vn_h^{k_\eps}$. This issue is tackled next.

\begin{theorem}[crtical points]\label{T:critical-points}
Let $\varepsilon\to0$ and let $\vn_h^{k_\varepsilon}$ be chosen from Corollary \ref{cor:small_residual} (residual estimate). Firstly, there are cluster points of $\{\vn_h^{k_\varepsilon}\}_{\varepsilon>0}$. Secondly, if $\vn^*_h$ is a cluster point of $\{\vn_h^{k_\varepsilon}\}_{\varepsilon>0}$, then it is a critical point of $ E$ over $\mathbb{V}_h$ in tangential directions, namely
\begin{equation}\label{eq:critical-point}
\frac{\delta E[\vn_h^{*}; \vv_h]}{\delta \vn} = 0
\end{equation}
for all $\vv_h\in \mathbb{T}(\vn_h^*)$.
\end{theorem}
\begin{proof}
We first note that we have the uniform bound $\Vert \nabla \vn_h^{k_\varepsilon}\Vert\leq \frac{2}{c_0}E[\vn_h^{0}]$ due to Lemma \ref{lem:equicoercivity-discrete} (equicoercivity) and the energy decreasing property of Theorem \ref{thm:energy_decrease} (energy stability and $L^\infty$ control of constraint). By compactness in the finite dimensional space $\mathbb{V}_h$, we deduce the existence of cluster points of $\{\vn_h^{k_\varepsilon}\}_{\varepsilon>0}$, namely the first claim.

If $\vn_h^* \in \V_h$ is a cluster point, we shall now prove that it is a critical point in the sense \eqref{eq:critical-point}. Let $\vv_h\in \mathbb{T}(\vn_h^*)$ and consider the discrete function  $\vphi_h := I_h\big[|\vn_h^*|^{-2}(\vv_h\times\vn_h^*)\big]\in\mathbb{V}_h$, which is well-defined because $|\vn_h^*(z)|\geq 1$ for all nodes $z\in \mathcal{N}_h$. Note that $\vn_h^*(z)\times \vphi_h(z) = \vv_h(z)$ at each $z\in \mathcal{N}_h$ because $\vv_h\in \mathbb{T}(\vn^*_h)$ and the cross product identity $\va\times(\vb\times\vc) = (\va\cdot\vc)\vb - (\va\cdot\vb)\vc$. Moreover, $I_h[\vn_h^{k_\varepsilon}\times \vphi_h]\in  \mathbb{T}(\vn_h^{k_\varepsilon})$.  

 Consider a subsequence $\vn_h^{k_\varepsilon}\to \vn_h^*$ as $\varepsilon\to0$ in any norm because $\mathbb{V}_h$ is finite dimensional. Then $I_h[\vn_h^{k_\varepsilon}\times \vphi_h]\to \vv_h$ as $\eps\to0$ as well as
\[
\frac{\delta E[\vn_h^{k_\varepsilon}; I_h[\vn_h^{k_\varepsilon}\times \vphi_h]]}{\delta \vn} \to \frac{\delta E[\vn_h^*; \vv_h]}{\delta \vn}
\]
because $\frac{\delta E}{\delta \vn}$ is continuous in each argument. Also, Corollary \ref{cor:small_residual}
(residual estimate) gives
\[
\left|\frac{\delta E[\vn_h^{k_\varepsilon}; I_h[\vn_h^{k_\varepsilon}\times \vphi_h]]}{\delta \vn}\right|\leq (1+\tau c_0+\tau c_1)\sqrt{2\varepsilon} \Vert \nabla I_h[\vn_h^{k_\varepsilon}\times \vphi_h]\Vert\leq C\sqrt{\varepsilon}\to0,
\]
whence
\begin{equation*}
\frac{\delta E[\vn_h^{*}; \vv_h]}{\delta \vn} = 0
\end{equation*}
for all $\vv_h\in T(\vn_h^*)$. This completes the proof.
\end{proof}

\begin{remark}[cross product]
\rm  
The above trick of the cross product to avoid dealing with the tangent space $\mathbb{T}(\vn_h^{k_\eps})$ has been used before in both numerical analysis and analysis of related problems \cite{alouges1997new, antil2021approximation, bartels2005stability, chen1989weak}. It hinges on the strong convergence of both $\vn_h^{k_\eps}$ and $I_h\big[ \vn_h^{k_\eps} \times \phi_h \big]$, which is true for $h$ fixed because $\V_h$ is finite dimensional. However, this argument does not extend to showing that a {\it discrete} critical point of $E$ converges to a {\it continuous} critical point as $h\to0$ as in \cite{bartels2005stability}.
This is because the product of two weakly convergent sequences may not converge weakly, which becomes an issue for the quartic terms of $E$.
\end{remark}

\subsection{Practical implementation: Lagrange multiplier}\label{sec:FO-practical-implementation}

To practically implement the gradient flow step in \eqref{eq:grad-flow-eq} we introduce a Lagrange multiplier $\lambda_h\in \mathbb{Q}_{h,0}$, the space of scalar continuous piecewise linear functions that vanish on $\partial\Omega$, and the bilinear form for the linear constraint $\vu_h\in \mathbb{T}_h(\vn^{k}_h)$, i.e.\ $(\vn^{k}_h(z)\cdot\vu_h(z) = 0)$
\begin{equation*}
b^k(\lambda_h,\vv_h) = \int_\Omega I_h[\lambda_h (\vv_h\cdot\vn_h^{k})];
\end{equation*}
this is a mass lumped $L^2$ inner product between $\lambda_h$ and $\vv_h\cdot\vn_h^{k}$ that depends on $k$. The gradient flow step is solved as a saddle point system:
\begin{align}
\label{eq:dtn-eq} a(d_t\vn^{k+1}_h,\vv_h)+b^k(\lambda_h, \vv_h) &= \langle \vf^{k},\vv_h\rangle & \forall \vv_h\in \mathbb{V}_{h,0} \\
\label{eq:lagrange-eq}b^k(\rho_h, d_t\vn^{k+1}_h) &= 0 &\forall \rho_h\in \mathbb{Q}_{h,0},
\end{align}
where
\begin{equation*}
a(\vu_h,\vv_h):= (1+c_0\tau)(\nabla \vu_h,\nabla \vv_h)+c_1\tau(\divrg \vu_h,\divrg \vv_h)
\end{equation*}
and
\begin{align*}
\langle \vf^{k},\vv_h\rangle:= -\frac{\delta  E[\vn^{k}_h;\vv_h]}{\delta \vn}.
\end{align*}
The saddle point system in \eqref{eq:dtn-eq} and \eqref{eq:lagrange-eq} is well-posed. First, the bilinear form $a$ is coercive over $\mathbb{V}_{h,0}$, and hence is coercive over the kernel of $b^k$. The bilinear form $b^k$ satisfies the following $h$-dependent and potentially suboptimal inf-sup inequality. We point to \cite{hu2009saddle} and \cite[Lemma 3.1(i)]{bartels2022quasi} for a uniform inf-sup property for $b^k$ measured in different norms.
\begin{prop}[inf-sup for linearized constraint]
Let $\tau h^{-1}\leq \Cstab$ and $|\vn^{0}_h(z)| = 1$ for all nodes $z\in\mathcal{N}_h$ as in Theorem \ref{thm:energy_decrease}. Let $\vn^{k}_h\in \mathbb{V}_h$ be the $k$-th iterate generated by Algorithm \ref{alg:grad-flow}. Then the bilinear form $b^k:\mathbb{Q}_{h,0}\times \mathbb{V}_{h,0}$ satisfies the following inf-sup inequality
\begin{equation}\label{eq:L2-inf-sup}
\inf_{\lambda_h\in \mathbb{Q}_{h,0}\setminus\{0\}}\sup_{\vv_h\in \mathbb{V}_{h,0}\setminus\{0\} }\frac{b^k(\lambda_h,\vv_h)}{\Vert \vv_h\Vert_{H^1}\Vert \lambda_h\Vert} \geq c h,
\end{equation} 
where $c>0$ only depends on $ E[\vn^{0}_h], \Cstab, c_{inv}$ and shape regularity and quasiuniformity of the sequence of triangulations $\{\mathcal{T}_h\}_h$.
\end{prop}
\begin{proof}
It suffices to prove that given a $\lambda_h\in \mathbb{Q}_{h,0}$ there exists $\vv_h\in \mathbb{V}_{h,0}$ such that $b^k(\lambda_h,\vv_h)\geq c h \Vert \vv_h\Vert_{H^1}\Vert \lambda_h\Vert$.

Let $\lambda_h\in \mathbb{Q}_{h,0}$ and choose $\vv_h = I_h[\lambda_h\vn_h^{k}]\in \mathbb{V}_{h,0}$. At each node, $z\in \mathcal{N}_h$, we have 
\[
\lambda_h(z)\vv_h(z)\cdot \vn_h^{k} = |\lambda_h(z)|^2 |\vn_h^{k}(z)|^2.
\]
Recall that Algorithm \ref{alg:grad-flow} produces $|\vn_h^{k}(z)|^2\geq 1$ at each node according to Remark \ref{rmk:lower-bound-length} (lower bound on $|\vn_h^{k}(z)|^2$). Hence, $\lambda_h(z)\vv_h(z)\cdot \vn_h^{k} \geq \lambda_h(z)^2$, and there is a constant $c>0$ independent of $h$ such that
\[
b^k(\lambda_h,\vv_h) =\int_\Omega I_h[\lambda_h (\vv_h\cdot\vn_h^{k})]\; d\vx \geq \int_\Omega I_h[\lambda_h^2]\; d\vx\geq c\Vert \lambda_h\Vert^2
\]
by virtue of the norm equivalence $\Vert I_h[\lambda_h]\Vert \approx \Vert \lambda_h\Vert$ on $\mathbb{Q}_{h,0}$. 

We are left to show $\Vert \vv_h\Vert_{H^1}\leq c h^{-1}\Vert\lambda_h\Vert$. Let $T\in\mathcal{T}_h$ be arbitrary
and notice that on $T$
\[
\nabla \vv_h = \nabla I_h\big[ \lambda_h \vn_h^k - \vc_T \big]
\]
where $\vc_T\in\mathbb{R}^3$ is a suitable constant. Applying local inverse and interpolations estimates, as well as the local stability of the Lagrange interpolation operator in $L^\infty$,  yields
\begin{align*}
  \| \nabla \vv_h \|_{L^2(T)} & \lesssim h_T^{-1+\frac32} \| I_h \big[ \lambda_h \vn_h^k - \vc_T \big] \|_{L^\infty(T)}
  \\
  &\le h_T^{\frac12} \|\lambda_h \vn_h^k - \vc_T \|_{L^\infty(T)}
  \\
  & \lesssim h_T^{\frac32} \|\nabla\lambda_h\|_{L^\infty(T)} \|\vn_h^k\|_{L^\infty(T)} +
  h_T^{\frac32} \|\lambda_h\|_{L^\infty(T)} \|\nabla\vn_h^k\|_{L^\infty(T)}
\end{align*}
upon taking $\vc_T$ to be the meanvalue of $\lambda_h \vn_h^k$ in $T$. In view of Theorem \ref{thm:energy_decrease} (energy stability and $L^\infty$ control of constraint), we deduce
\[
\|\vn_h^k\|_{L^\infty(T)} \lesssim 1+2c_{inv} \sqrt{\Cstab E[\vn_h^0]} =: c
\quad\Rightarrow\quad
\|\nabla \vn_h^k\|_{L^\infty(T)} \lesssim c\, h_T^{-1} 
\]
by a local inverse estimate. We then apply local inverse estimates on $\lambda_h$ 
to deduce
\[
\| \nabla \vv_h \|_{L^2(T)} \lesssim  c\, h_T^{-1}  \, \|\lambda_h\|_{L^2(T)}.
\]
Squaring, adding over $T\in\mathcal{T}_h$ and using that $\mathcal{T}_h$ is quasi-uniform gives \eqref{eq:L2-inf-sup}.
\end{proof}
\begin{remark}  \rm
  Our inf-sup condition in \eqref{eq:L2-inf-sup} is proportional to $h$ because the norm $\|\lambda_h\|$ of
    the multiplier is $L^2(\Omega)$ rather than $H^{-1}(\Omega)$. In \cite{bartels2022quasi,hu2009saddle} a uniform
    inf-sup constant of the form $c \|\nabla \vn_h^k\|_{L^\infty}^{-1}$ is derived for harmonic maps provided
    $|\vn_h^k(z)|=1$ for all $z\in\mathcal{N}_h$. Two comments are in order. First, the inf-sup constant is
    mesh-independent provided $\|\nabla\vn_h^k\|_{L^\infty}\le C$, but this precludes the occurrence of defects
    whose capture and approximation is one of the highlights of this paper. Second, the proof relies on
    enforcing the unit length constraint of $\vn_h^k$ at nodes, which is against the relaxed condition
    $\norm{I_h [ |\vn_h^k|^2 - 1 ]} \le \eta$ assumed in definition \eqref{eq:discrete-admin-set}
    of the discrete admissible set $\mathcal{A}_{\vg,h,\eta}$.
    
\end{remark}
\begin{remark}[Newton iteration]
  \rm If we were to implement Newton's method to find critical points of $ E$ over $\mathcal{A}_{\vg,h,0}$, the linear system for the Newton iterates $d_t\vn^{k+1}_h = \vn^{k+1}_h - \vn_h^{k}$ and $d_t\lambda^{k+1}_h = \lambda^{k+1}_h - \lambda_h^{k}$ would read
\begin{align*}
\frac{\delta^2 E[\vn^k_h; d_t\vn^{k+1}_h, \vv_h]}{\delta \vn^2} + b^k(d_t\lambda_h^{k+1}, \vv_h) &= - \frac{\delta  E[\vn^k_h; \vv_h]}{\delta \vn} - b^k(\lambda_h^{k}, \vv_h)\\
b^k(\rho_h, d_t\vn^{k+1}_h) &= -\frac{1}{2}\int_\Omega I_h\Big[\big(|\vn^k_h|^2-1\big) \rho_h\Big] \; d\vx.
\end{align*}
The above system has a similar structure to the system \eqref{eq:dtn-eq} and \eqref{eq:lagrange-eq}. First, the form $b^k$ would satisfy the same $h$ dependent inf-sup condition. Second, it is not clear whether $\delta^2E[\vn^k_h; \cdot, \cdot]/\delta \vn^2$ is coercive. One needs to find an energy equivalent to $E$ to ensure coercivity. This was done in \cite{adler2015energy}, which shows that an appropriate modification of $E$ leads to coercivity of the second variation for $k_2/k_3\in (1-\varepsilon_k,1+\varepsilon_k)$ for some $\varepsilon_k$ that depends on $\vn_h^k$. However, \cite[Remark 3.9]{adler2015energy} points out that the bound on $\varepsilon_k$ goes to 0 as $\Vert\nabla\vn_h^k\Vert_{L^\infty}\to\infty$. As a result, we might expect to lose coercivity with mesh refinement if $k_2\neq k_3$ and if there are defects present in the liquid crystal. 
\end{remark}

For solving the saddle point system, we use MINRES \cite{paige1975solution}. 

\section{Magnetic effects}\label{sec:magnetic}

This section addresses how to adjust the preceding theory to the presence of a {\it fixed} magnetic field $\vH$. 
For  $\vH\in L^2(\Omega;\mathbb{R}^3)$ the magnetic energy is \cite[Ch. 4.1]{virga1995variational}
\begin{equation*}
\Emag[\vn] = -\frac{\chi_A}{2}\int_{\Omega}(\vn\cdot\vH)^2d\vx
\end{equation*}
where $\chi_A$ is the diamagnetic anisotropy, which measures how much a liquid crystal wants to either align with the magnetic field or align orthogonally to the magnetic field. The parameter $\chi_A$ may be positive or negative depending on the material. For this paper, we consider $\chi_A\geq 0$, which favors alignment of $\vn$ with $\vH$.

With the magnetic energy, the total energy becomes 
\begin{equation*}
\Etotal[\vn] :=  E[\vn]+\Emag[\vn]
\end{equation*}
Since the magnetic contribution is a lower order term, existence of minimizers is still true \cite[Theorem 2.3]{hardt1987mathematical}. We now summarize the numerical results in the presence of the extra magnetic field $\vH$ and remark on how the proofs are modified.

\subsection{Convergence of minimizers}
%
The following statement is a complement to Theorem \ref{thm:conv-min} (convergence of minimizers).
\begin{theorem}[convergence of minimizers with magnetic field]
Let $h\to0$. There exists a sequence $\{\eta_h\}_h$ with $\eta_h\to0$ as $h\to0$ such that the sequence of minimizers $\vn^*_{h,\eta_h}$ of $\Etotal$ over the admissible set $\mathcal{A}_{\vg, h,\eta_h}$ admits a subsequence (not relabeled) $\vn^*_{h,\eta_h}$ such that $\vn^*_{h,\eta_h}\rightharpoonup \vn^*$ in $H^1(\Omega;\mathbb{R}^3)$ and $\vn^*\in \mathcal{A}_{\vg}$ is a minimizer of $\Etotal$ over $\mathcal{A}_\vg$. Moreover, $\Etotal[\vn^\ast_{h,\eta_h}]\to \Etotal[\vn^\ast]$ as $h\to0$.

\end{theorem}
\begin{proof}
We simply adjust the proofs of Lemmas \ref{lem:recovery-sequence} (recovery sequence), \ref{lem:equicoercivity-discrete} (equicoercivity), and \ref{lem:discrete-wlsc} (weak lower semicontinuity), exploiting the fact that $\Emag[\vn]$ is a lower order perturbation of $E[\vn]$ for $\vH$ fixed. For both Lemmas \ref{lem:recovery-sequence} and \ref{lem:discrete-wlsc}, we can extract a subsequence $\{\vn_h\}_h$ (not relabeled) such that $\vn_{h}\to\vn$ strongly in $L^2$ and a.e. in $\Omega$. Since $\vn_h$ is uniformly bounded in $L^\infty$, the Lebesgue Dominated Convergence Theorem implies that $\Emag[\vn_{h}] \to \Emag[\vn]$, and both the recovery sequence and liminf arguments carry over. For Lemma \ref{lem:equicoercivity-discrete}, the uniform bound in $L^\infty$ in the definition of $\mathcal{A}_{\vg,h,\eta}$ ensures that $\Emag[\vn_{h,\eta}]\geq -\frac{C|\chi_A|}{2}\Vert \vH\Vert^2$, which does not impact equicoercivity of $\Etotal$.
\end{proof}

\subsection{Gradient flow}
We need a slight modification of Algorithm \ref{alg:grad-flow} (projection-free gradient flow): since $\chi_A\geq0$, we treat $\Emag$ explicitly to guarantee energy decrease. The resulting scheme reads as follows: find $d_t\vn^{k+1}_h \in \mathbb{T}_h(\vn^{k}_h)$ that solves
\begin{equation}\label{eq:grad-flow-magnetic}
(1+\tau c_0)(\nabla d_t\vn^{k+1}_h, \nabla \vv_h) +\tau c_1(\divrg d_t\vn^{k+1}_h, \divrg \vv_h) = - \frac{\delta \Etotal[\vn_h^{k};\vv_h]}{\delta \vn}
\end{equation}
for all $\vv_h\in \mathbb{T}_h(\vn^{k}_h)$ instead of \eqref{eq:grad-flow-eq}. We then have the following complement to Theorem \ref{thm:energy_decrease} (energy stability and $L^\infty$ control of constraint).
\begin{theorem}[energy decrease and $L^\infty$ control of constraint with magnetic effects]
  Let $\vn^{0}_h\in \mathbb{V}_h$ such that $|\vn^{0}_h(z)|^2= 1$ for all $z\in \mathcal{N}_h$. There is a constant $0<\Cstab\leq1$ which may depend on $\Etotal[\vn^{0}_h], c_{inv},$ and $c_i$ for $i=0,1,2,3$ such that if $\tau h^{-1}\leq \Cstab$ then
\[
\Etotal[\vn^{k+1}_h] +\frac{\tau}{2} \sum_{\ell=0}^k\Vert \nabla d_t\vn^{\ell+1}_h\Vert^2 \leq \Etotal[\vn^{0}_h]
\quad\forall k
\]
and
\[
\big| |\vn^{k+1}_h(z)|^2-1 \big| \leq 4c_{inv}^2 \tau h^{-1} \Etotal[\vn^{0}_h]
\quad\forall z\in \mathcal{N}_h,
\]
where $c_{inv}$ is the constant from Lemma \ref{lem:discrete-sob} (discrete Sobolev inequality).
\end{theorem}
\begin{proof}
The explicit treatment of $\Emag$ guarantees energy decrease due to the quadratic identity $(\va,\va-\vb) = \frac{1}{2}\norm{\va}^2 - \frac{1}{2}\norm{\vb}^2 +\frac{1}{2}\norm{\va-\vb}^2$. In fact, we have
\begin{align*}
- \frac{\delta E_m[\vn^{k}_h;\tau d_t\vn^{k+1}_h]}{\delta \vn} &= -\frac{\chi_A}{2}\|\vn^{k}_h\cdot\vH\|^2 +\frac{\chi_A}{2}\|\vn^{k+1}_h\cdot\vH\|^2 -\frac{\tau^2\chi_A}{2}\|d_t\vn^{k+1}_h\cdot\vH\|^2\\
&=\Emag[\vn^{k}_h] - \Emag[\vn^{k+1}_h]-\frac{\tau^2\chi_A}{2}\|d_t\vn^{k+1}_h\cdot\vH\|^2.
\end{align*}
Inserting this low order perturbation in the proof of Theorem \ref{thm:energy_decrease} does not alter the derivation of the energy bound. Moreover, the length constraint is not directly related to $\Emag$.
\end{proof}
Corollaries \ref{cor:L1-constraint} (control of $L^1$ violation of constraint) and \ref{cor:small_residual} (residual estimate) as well as Theorem \ref{T:critical-points} (critical points) in Section \ref{sec:proj-free} also extend to magnetic fields.

\section{Computational results}\label{sec:computation}

Our algorithm was implemented\footnote{Code implementing Algorithm \ref{alg:grad-flow} is available at \href{https://github.com/LBouck/FullFrankOseen-2024}{https://github.com/LBouck/FullFrankOseen-2024} } 
in the multi-physics software NGSolve \cite{schoberl2014c++} and visualizations were made with ParaView \cite{ahrens2005paraview}. Since we are interested in the influence of Frank's constants, we introduce the following notation to indicate the three main effects of \eqref{E:FO}:
\begin{align*}
\mathsf{splay}[\vn] := \int_\Omega (\divrg\vn)^2d\vx, \quad \mathsf{twist}[\vn] := \int_\Omega (\vn\cdot\curl\vn)^2d\vx,\quad \mathsf{bend}[\vn] := \int_\Omega (\vn\times\curl\vn)^2d\vx.
\end{align*}
We also use short hand notation for the $L^p$ discrete unit length constraint error
\begin{equation*}
\mathsf{err}_p(\vn_h) := \|I_h[|\vn_h|^2-1]\|_{L^p}.
\end{equation*}
Finally $\vn_h^{\infty}$ will denote the solution produced by Algorithm \ref{alg:grad-flow} when the desired tolerance is reached.

\subsection{Frank's constants and defects}

This section investigates how defects may change behavior under the influence of Frank's constants $k_1$ (splay), $k_2$ (twist), $k_3$ (bend). The first example is the instability of the degree-one defect $\vn_1(\vx) = \vx/|\vx|$ for $k_2$ sufficiently small relative to $k_1$ and $k_3$, known as H\'{e}lein's condition. The second example is the instability of a degree-two defect and the influence of Frank's constants on the resulting configuration.

\subsubsection{H\'{e}lein's condition}

The second variation of $E$ over $\mathcal{A}_\vg$ with
  $\Omega = B_1(0)$ at the degree-one defect $\vn_1(\vx)$ is known to be positive definite if and only if the
  H\'{e}lein's condition \cite{kinderlehrer1992second}
\begin{equation}\label{eq:helene-condition}
8(k_2 - k_1)+k_3\geq 0,
\end{equation}
is satisfied. This computational experiment explores the structure of the solution $\vn_h$ when \eqref{eq:helene-condition} is violated. Note that $\mathsf{twist}[\vn_1]=\mathsf{bend}[\vn_1]=0$, so $\vn_1$ consistes of pure splay.
H\'{e}lein's condition quantifies the tradeoff between splay, bend and twist energies. If bend and twist constants $k_2,k_3$, are small relative to splay constant $k_1$, then it is energetically favorable for a configuration to bend and twist a little; \eqref{eq:helene-condition} does not hold. If $k_1$ is small relative to $k_2,k_3$, then \eqref{eq:helene-condition} is valid and the energy cannot reduce by bending and twisting, and splay is the preferred configuration.

For the next set of computations, we let Frank constants be
$$
k_1 = k_3 = 1,\quad k_2 = 0.1,
$$
which correspond to violation of \eqref{eq:helene-condition}; we thus expect $\vn_1(\vx)$ not to be a minimizer. We set the initial condition of the gradient flow to be $\vn^0_h = I_h \vn_1$ and set the discretization parameters to be $\tau = h = 2^{-\ell/2}$ for  $\ell = 4,\ldots, 9$ to see how the projection-free gradient flow behaves when decreasing $h$ and $\tau$. Table \ref{tab:quantitative-prop} shows the initial and final energy as well as gradient flow iteration counts; note that the number of gradient flow iterations grows like $\mathcal{O}(\tau^{-1})$. According to Theorem \ref{thm:energy_decrease} (energy stability and $L^\infty$ violation of constraint) and Corollary \ref{cor:L1-constraint} (control of $L^1$ violation of constraint), we expect $\textsf{err}_\infty(\vn_h^{\infty})\lesssim 1$ and $\textsf{err}_1(\vn_h^{\infty})\lesssim h$. Table \ref{tab:quantitative-errors} shows that $\textsf{err}_1(\vn_h^{\infty})\lesssim h$ and $\textsf{err}_\infty(\vn_h^{\infty})$ starts to decrease and perform slightly better than $\mathcal{O}(1)$.

\begin{figure}
\begin{center}
\includegraphics[width=.5\textwidth]{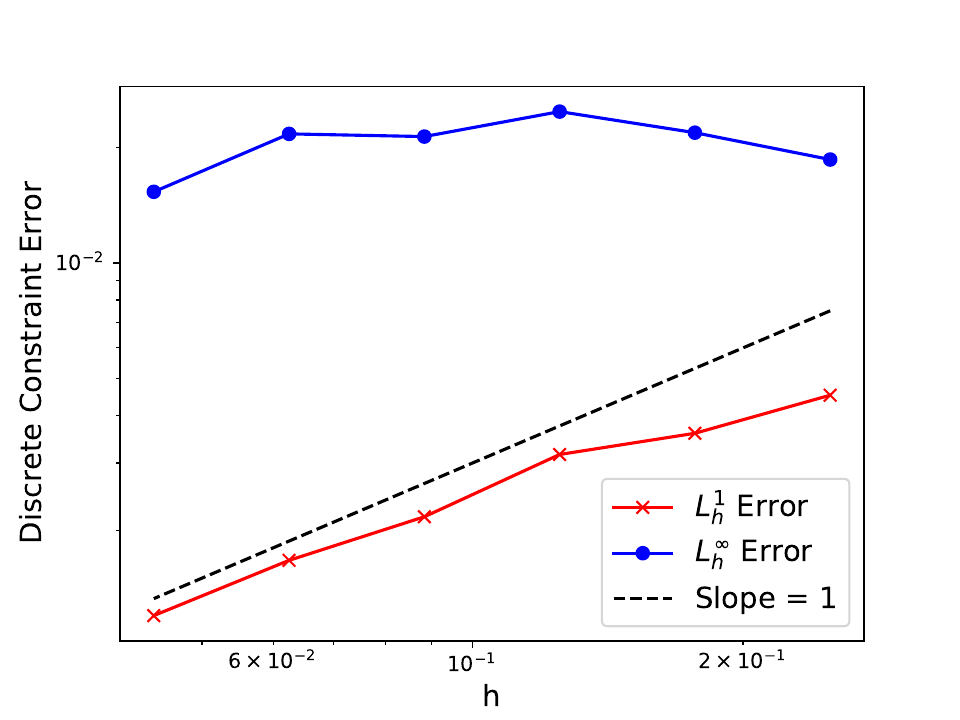}
\end{center}
\caption[Unit length constraint errors with respect to refinement in mesh size and pseudotime-step size]{Plot of discrete unit length constraint errors $\Vert I_h[|\vn_h^\infty|^2-1]\Vert_{L^p(\Omega)}$ for $p=1,\infty$. Note that Theorem \ref{thm:energy_decrease} and Corollary \ref{cor:L1-constraint} imply that $\textsf{err}_1(\vn_h^{\infty})\lesssim h$ and $\textsf{err}_\infty(\vn_h^{\infty})\lesssim 1$ provided $\tau h^{-1}\leq \Cstab$. The computational results corroborate the theory.} \label{fig:quantitative-prop}
\end{figure}

\begin{table}
\begin{center}
\begin{tabular}{ |c |c| c| c| }
\hline
$h$ 			& $ E[\vn_h^{0}] $ 	& $ E[\vn_h^{\infty}] $ 	& GF Iterations \\\hline
 $2^{-2}$ 		& 21.686 				& 21.147					&    397	\\  \hline
 $2^{-5/2}$ 	& 22.282 				& 21.557					&    178	\\  \hline
 $2^{-3}$ 		& 23.067				& 21.871 					&    248	\\  \hline
 $2^{-7/2}$ 	& 23.104 				& 21.925 					&    660	\\  \hline
 $2^{-4}$ 		& 23.480				& 21.988 					&    544	\\  \hline
 $2^{-9/2}$ 	& 23.521 				& 21.990 					&    794	\\  \hline

\end{tabular}
\end{center}
\caption[Gradient flow iterations with respect to refinement in mesh size and pseudotime-step size]{Table of initial energies, final energies and number of gradient flow iterations for decreasing values of $h$ to approximate a degree 1 defect with H\'{e}lein's condition \eqref{eq:helene-condition} being violated.} \label{tab:quantitative-prop}
\end{table}

\begin{table}
\begin{center}
\begin{tabular}{ |c |c| c| c| }\hline
$\ell$ & $\textsf{err}_1(\vn_h^{\infty}) $ & $\textsf{err}_\infty(\vn_h^{\infty}) $ & GF Iterations \\\hline
 1 & 1.15e-02 & 9.50e-02 & 69 \\ \hline
 2 & 6.07e-03 & 5.02e-02 & 129 \\ \hline
 3 & 3.13e-03& 2.59e-02 & 248 \\ \hline
 4 & 1.59e-03 & 1.31e-02 & 486 \\ \hline
 5 & 7.99e-04 & 6.58e-03 & 964 \\ \hline
 6 & 4.01e-04 & 3.30e-03 & 1918 \\ \hline
\end{tabular}
\caption[Discrete unit length constraint errors of Algorithm \ref{alg:grad-flow} vs.\ time refinement]{Discrete unit length constraint errors and number of gradient flow iterations for a degree 1 defect with H\'{e}lein's condition being violated with $h=1/8$, $\tau = \frac{1}{2^\ell}$ for $\ell = 1,\ldots, 6$ and $\varepsilon = 10^{-3}/2$. Both the $L^1$ and $L^\infty$ errors for the discrete unit length constraint decreases linearly with $\tau$, which is expected from Theorem \ref{thm:energy_decrease} and Corollary \ref{cor:L1-constraint} if $h$ is fixed. Also, gradient flow iterations increase like $\mathcal{O}(\tau^{-1})$, which is also expected if $\varepsilon$ is fixed.
}\label{tab:quantitative-errors}
\end{center}
\end{table}

For the smallest meshsize $h$ and time step $\tau$, we plot the initial and final configuration in Figure \ref{fig:helene}: we see that twist is preferred over splay and bend. We also display the initial and final splay, bend and twist energies in Table \ref{tab:init-final-energy-helene}. Note that twist increases by an order of magnitude while bend does it by one from the initial to final configuration.  This confirms the suspicion that the liquid crystal can decrease the energy by reducing splay at a modest cost of increasing twist and bend.
\begin{figure}[htb]
\begin{minipage}{.38\textwidth}
\begin{center}
\includegraphics[width=1\linewidth]{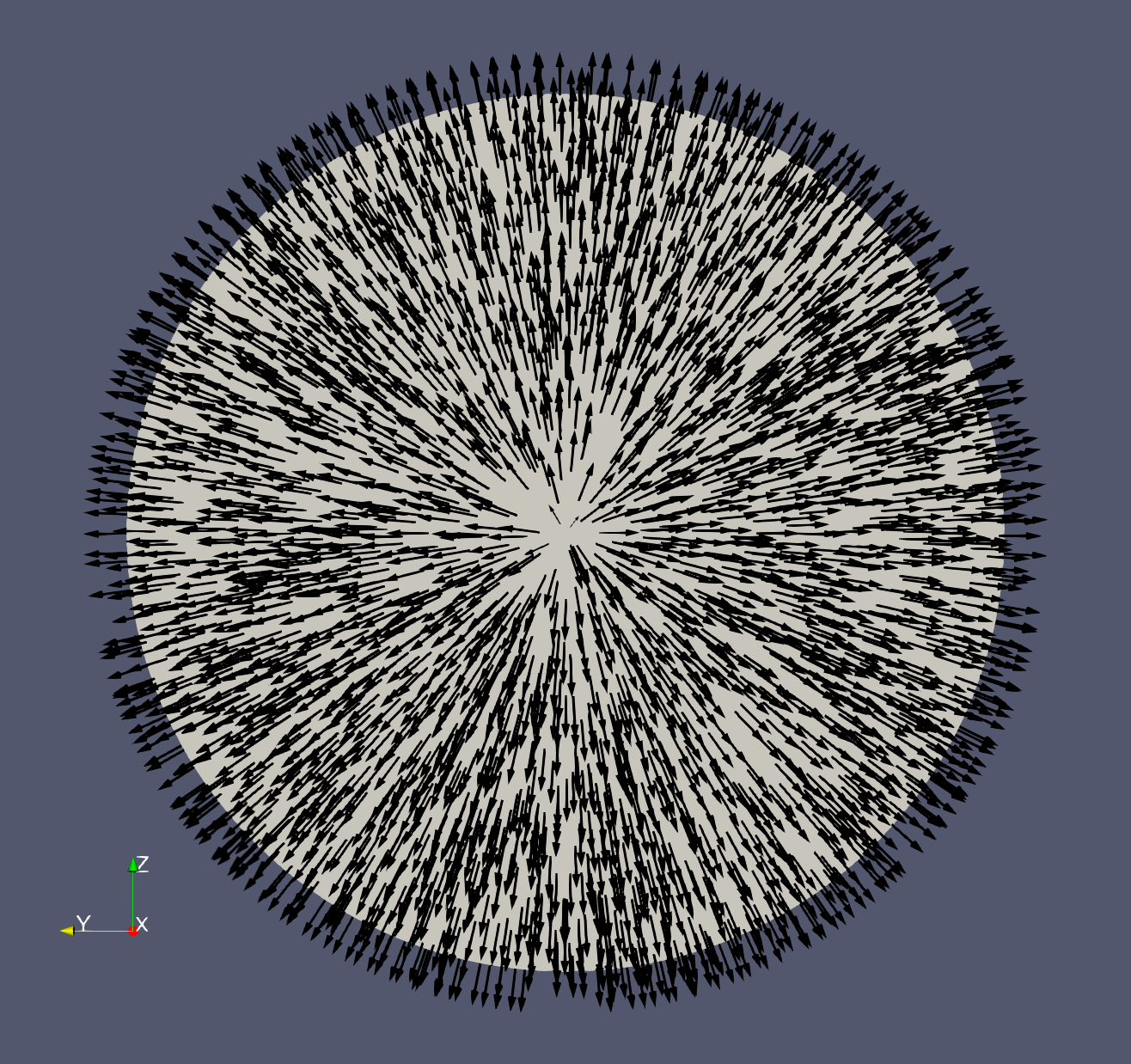}
\end{center}
\end{minipage}\hskip1.cm 
\begin{minipage}{.38\textwidth}
\begin{center}
\includegraphics[width=1\linewidth]{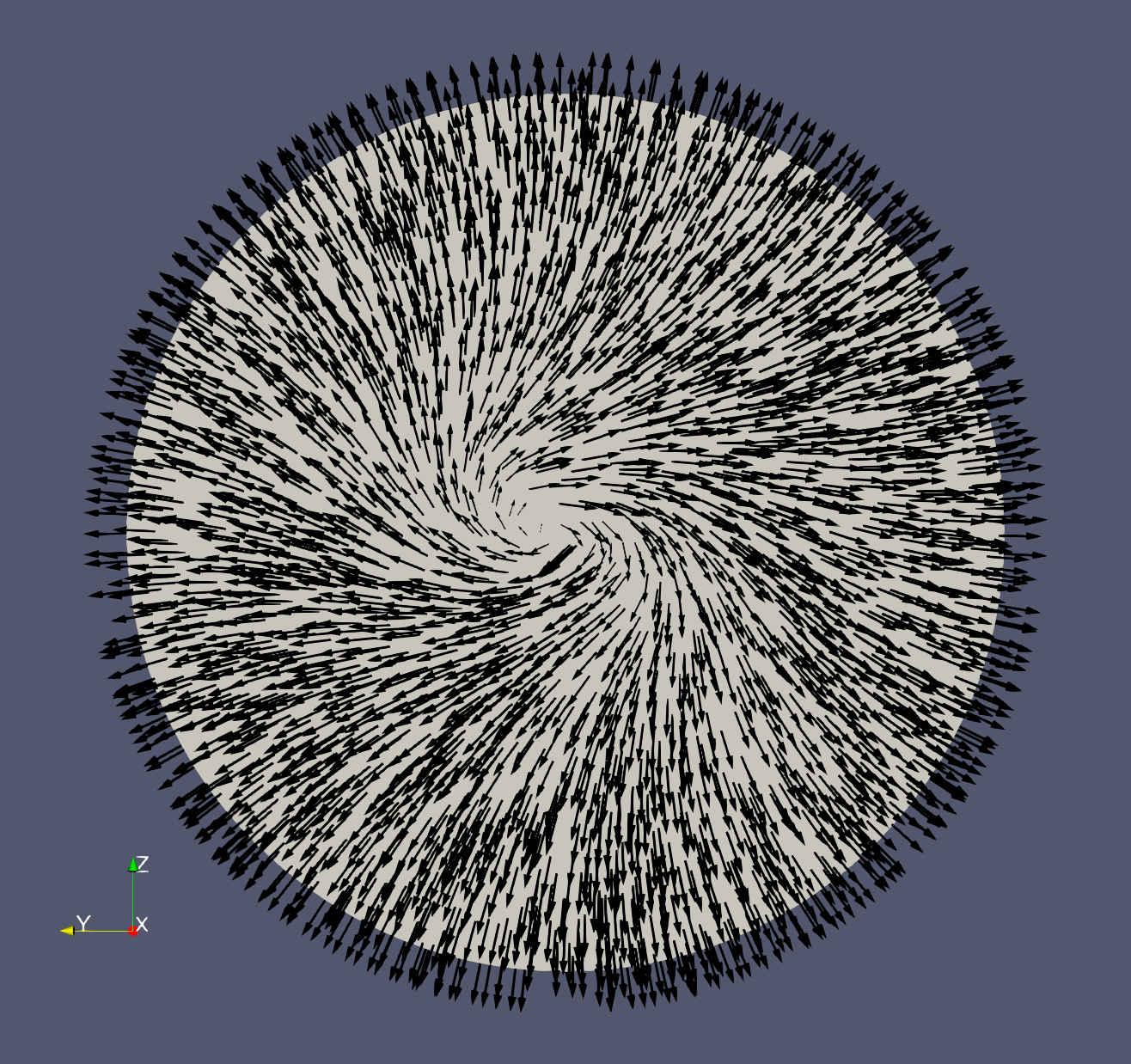}
\end{center}
\end{minipage}
\caption[Degree 1 defect under H\'{e}lene's condition]{Slice of the projected director field at $\{x=0\}$. Initial configuration (left) and computed minimizer (right) with $k_1=k_3=1$ and $k_2=.1$ and numerical parameters $h=\tau2^{-9/2}, \varepsilon= 10^{-3}/2$ (the stopping parameter of Algorithm \ref{alg:grad-flow}). Twist is preferred to splay and bend, in agreement with Helein's condition \eqref{eq:helene-condition} for $k_2$ sufficiently small relative to $k_1$ and $k_3$.
}\label{fig:helene}
\end{figure}
\begin{table}
\begin{center}
\begin{tabular}{ |c |c| c| c| }\hline
& $\textsf{splay}(\vn_h)$ & $\textsf{twist}(\vn_h)$ & $\textsf{bend}(\vn_h)$ \\\hline
 Initial & 49.4 & .0351 & .141 \\ \hline
 Final & 42.7 & 10.1 &    2.72\\  \hline
\end{tabular}
\caption[Initial and final components of Frank energy under H\'{e}lein's condition]{Initial and final splay, twist, and bend for computed solution with $k_1=k_3 = 1$ and $k_2=.1$ and $h=\tau=2^{-9/2}$ and $\varepsilon=10^{-3}/2$ (the stopping parameter of Algorithm \ref{alg:grad-flow}). Note that twist and bend increase by at least an order of magnitude while splay only decreases slightly. This sheds light on H\'{e}lein's condition being a tradeoff between splay and twist and bend.
}
\label{tab:init-final-energy-helene}
\end{center}
\end{table}

\subsubsection{Influence of Frank's constants on instability of degree 2 defect}

All simulations were computed with the following parameters
\begin{equation*}
\Omega = B_1(0), \quad h = \tau = \frac{1}{16}, \quad \varepsilon = 10^{-4}.
\end{equation*}
The starting configuration and boundary conditions is given by the degree 2 defect:
\begin{equation*}
\vn_2(\vx) = \pi^{-1}\left((\pi(\vx/|\vx|)^2\right),
\end{equation*}
where $\pi: \mathbb{S}^2\to \mathbb{C}$ is the stereographic projection. This example has been explored previously in the one constant case \cite{alouges1997new, bartels2005stability, cohen1987minimum}.

For all numerical simulations, the initial configuration is $\vn_h^{0} = I_h \vn_2$. Figure \ref{fig:deg2-oneconst} shows the initial condition and the result of the gradient flow for the one constant case $k_i=1$. Figure \ref{fig:deg2-k_3} shows the final configurations for $k_1=1,k_2 = .75$ and $k_3 = 1,3,5$. Note in Figure \ref{fig:deg2-k_3} that as $k_3$ increases, the computed solution transitions from two bending defects to two splay defects. In fact, as $k_3 = 1,3,5$, the value of $8(k_2 - k_1)+k_3$ from H\'{e}lein's condition is $-1, 1, 3$. When $8(k_2 - k_1)+k_3< 0$, we expect that bending and twist configurations are preferable based on H\'{e}lein's condition. Likewise for  $8(k_2 - k_1)+k_3\geq 0$, we expect splay configurations to be preferable. This heuristics is confirmed by the different configurations in Figure \ref{fig:deg2-k_3}.

\begin{figure}[htb]
\begin{minipage}{.38\textwidth}
\begin{center}
\includegraphics[width=1\linewidth]{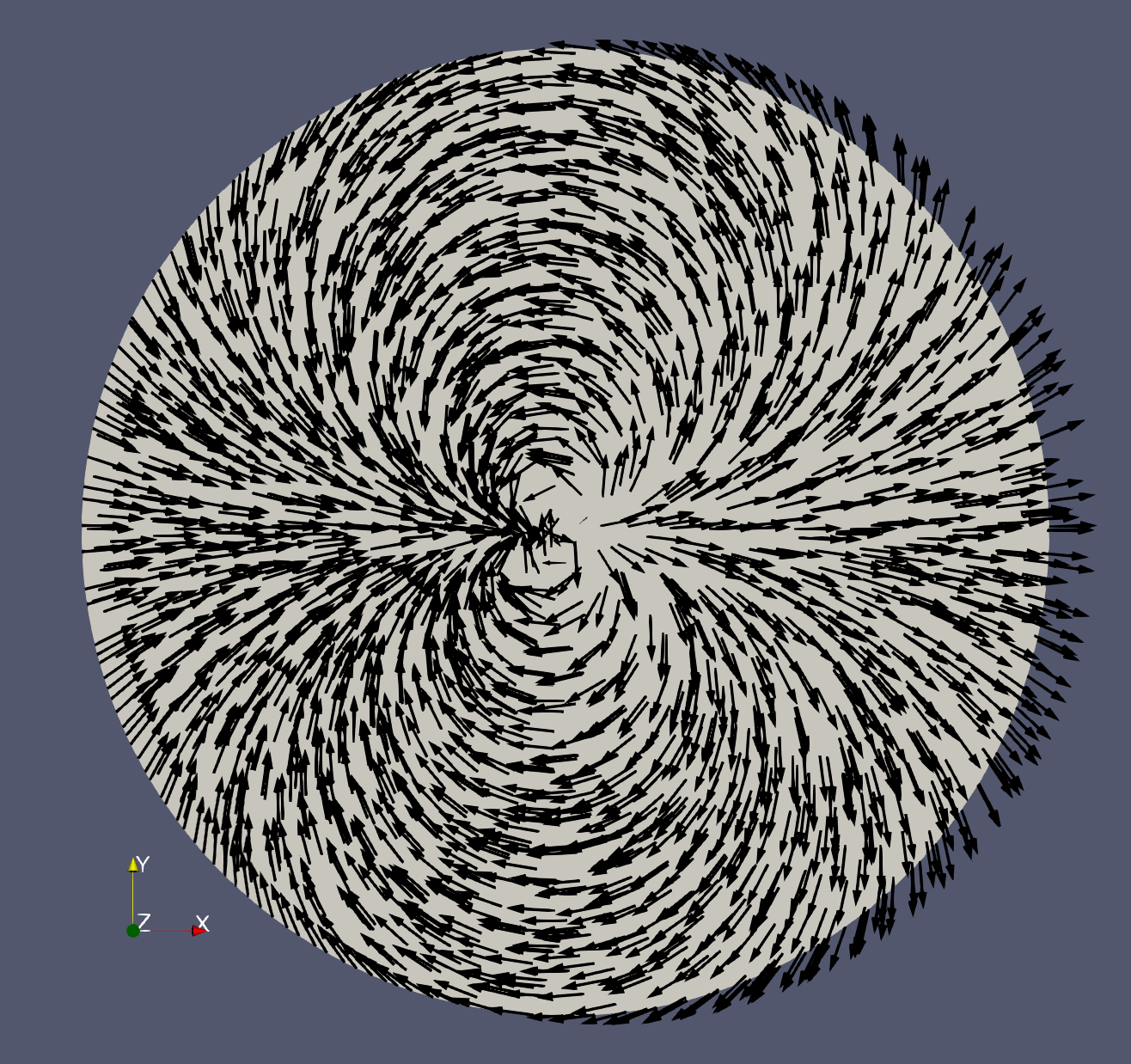}
\end{center}
\end{minipage}\hskip1.cm 
\begin{minipage}{.38\textwidth}
\begin{center}
\includegraphics[width=\linewidth]{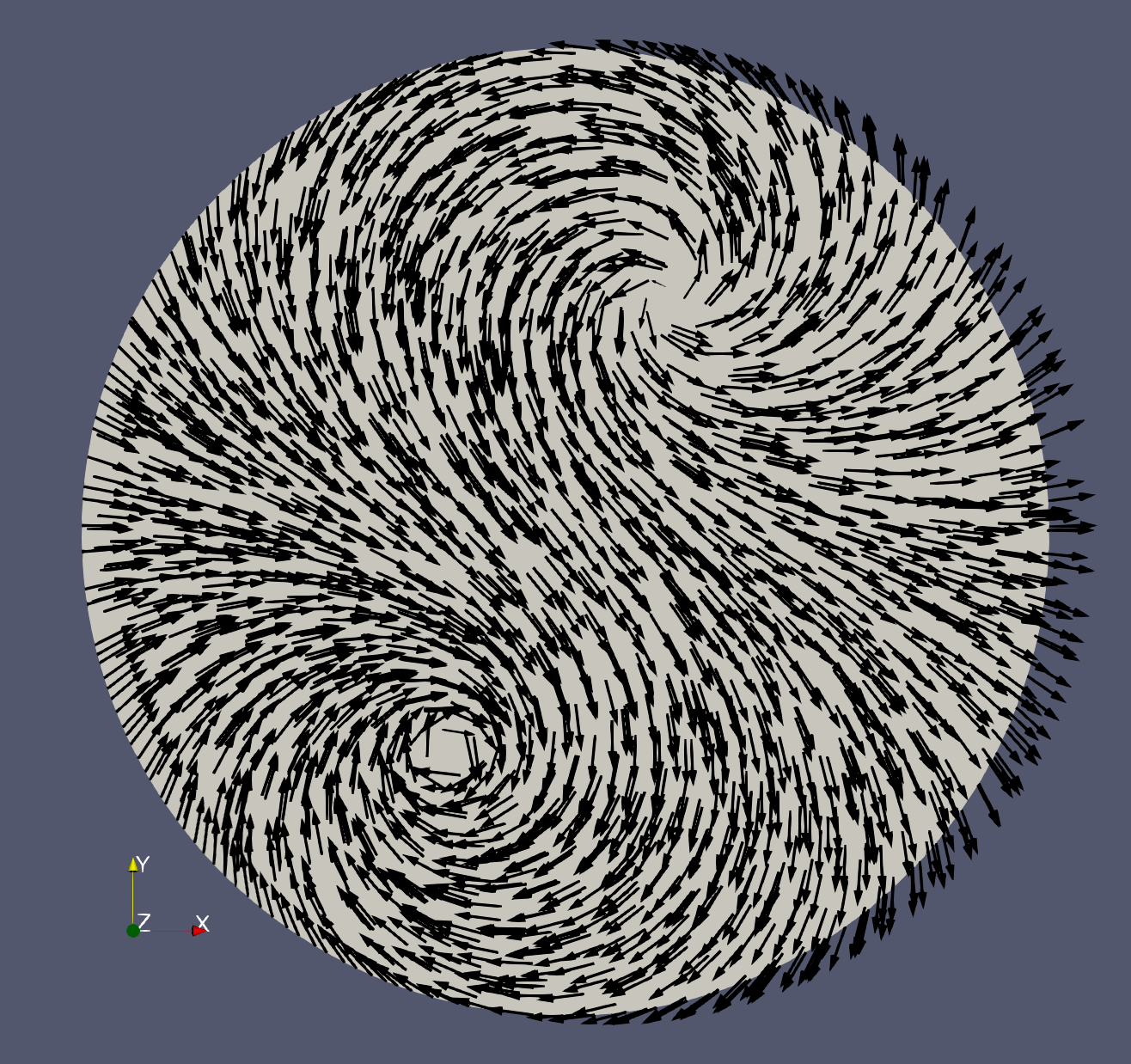}
\end{center}
\end{minipage}
\caption[Instability of degree 2 defect]{Initial and final projected director fields at $\{z=0\}$ from Algorithm \ref{alg:grad-flow} with $k_i=1$ for $i=1,2,3$, $h=\tau = 1/16$ and $\varepsilon = 10^{-4}$. This corresponds to the one-constant case $c_0=1,
    c_1=c_2=c_3=0$ in \eqref{eq:E-tilde2}. A degree 2 defect for the initial condition splits into two degree 1 defects.} \label{fig:deg2-oneconst}
\end{figure}
\begin{figure}[htb]
\begin{minipage}{.33\textwidth}
\begin{center}
\includegraphics[width=1\linewidth]{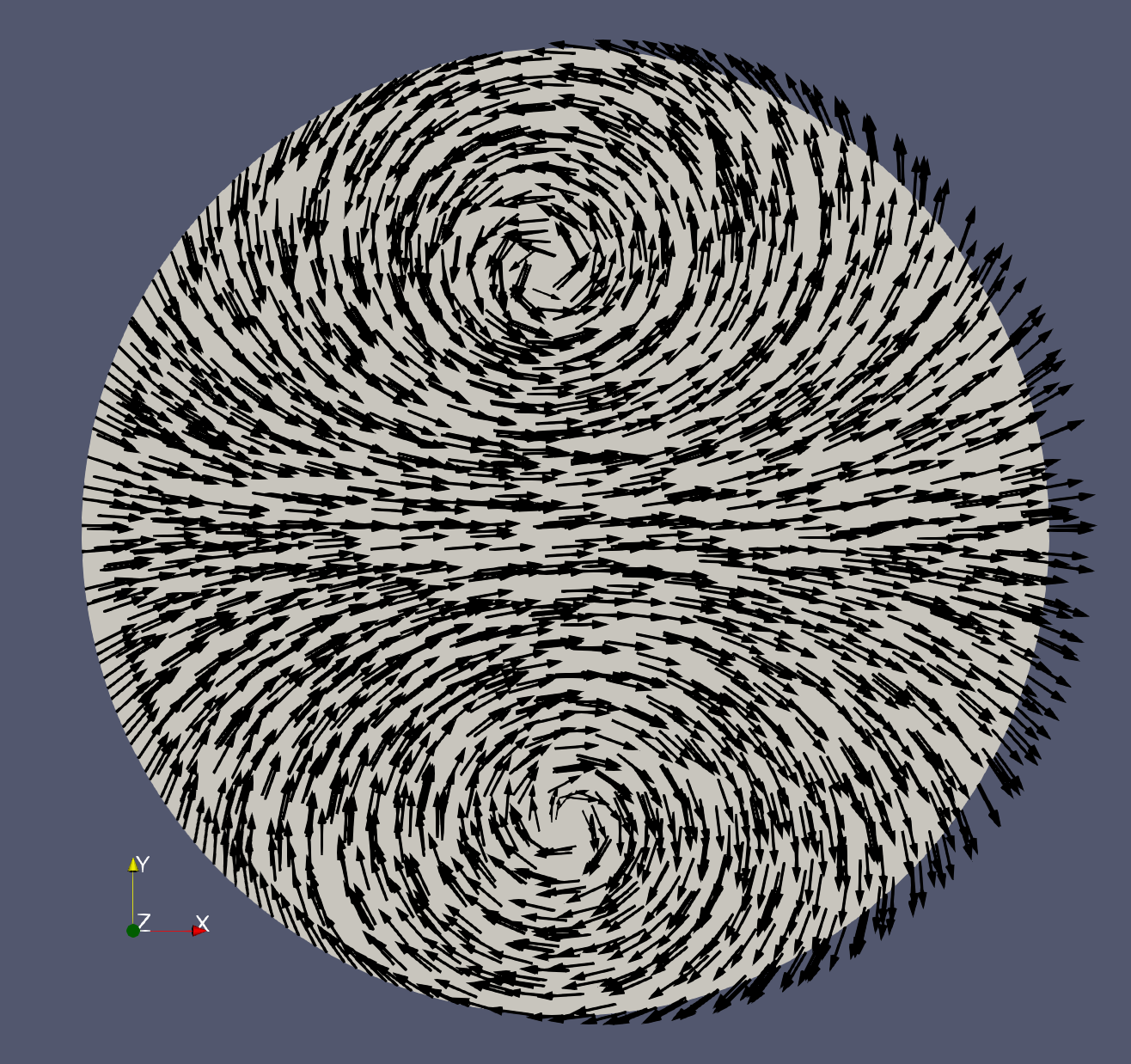}
\end{center}
\end{minipage}\hfill
\begin{minipage}{.33\textwidth}
\begin{center}
\includegraphics[width=1\linewidth]{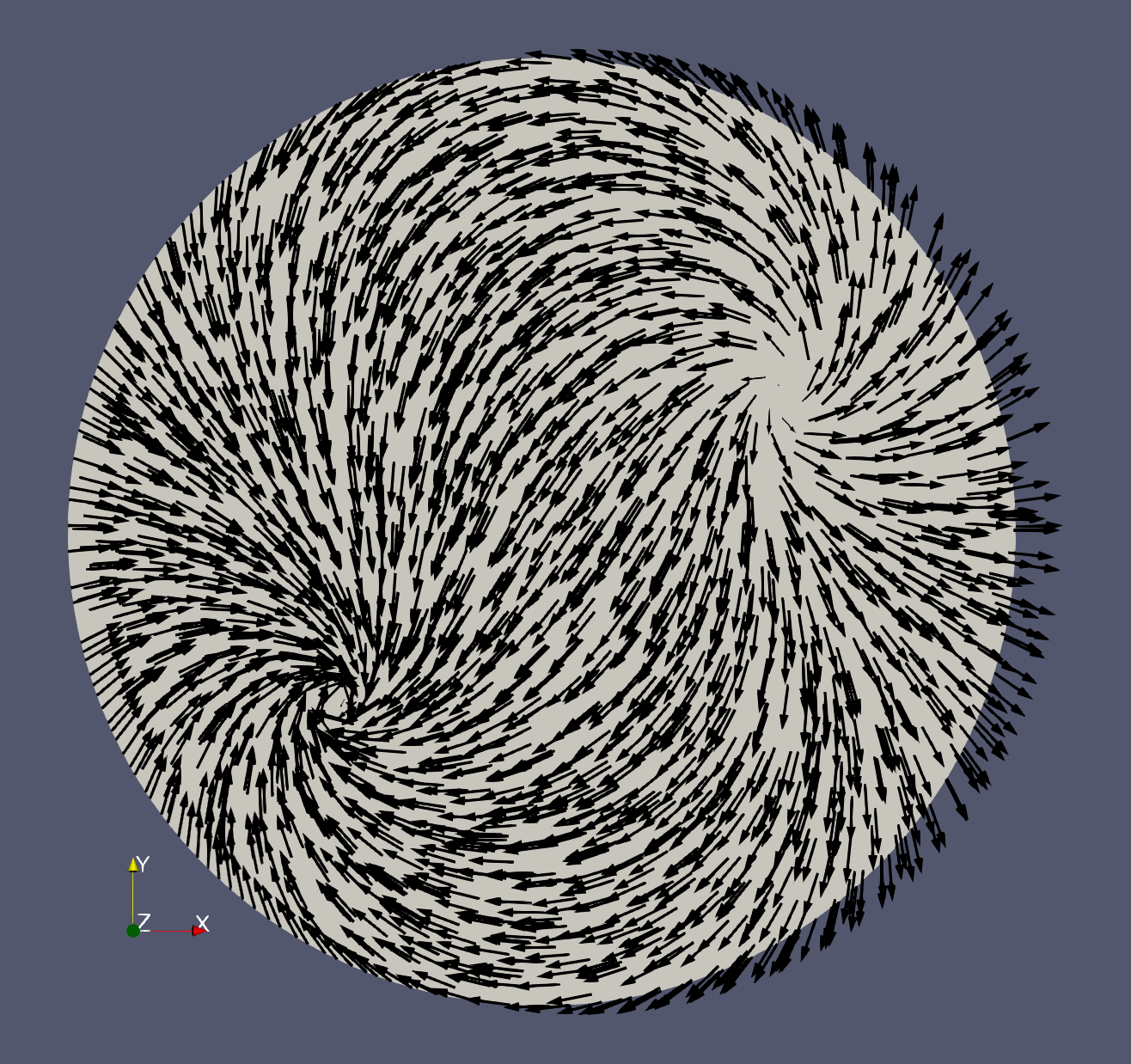}
\end{center}
\end{minipage}\hfill
\begin{minipage}{.33\textwidth}
\begin{center}
\includegraphics[width=1\linewidth]{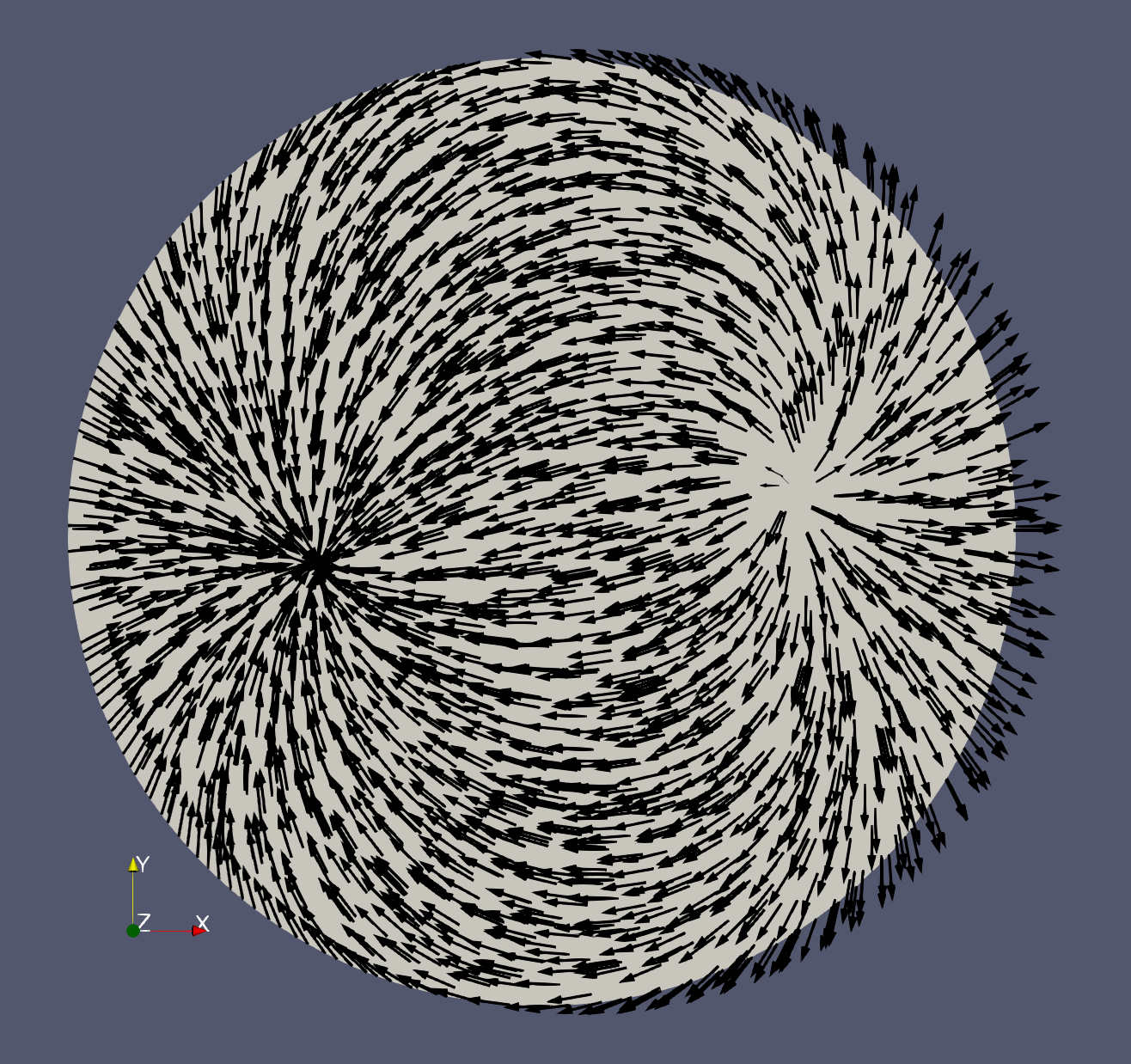}
\end{center}
\end{minipage}
\caption[Influence of Frank's constants on degree 2 defect]{Final projected director fields at $\{z=0\}$ from Algorithm \ref{alg:grad-flow} and $k_1=1,k_2=.75$ and $k_3 = 1,3,5$. The equilibrium configuration changes from two bending degree 1 defects for $k_3=1$ to two degree 1 splay like defects for $k_3=5$.}\label{fig:deg2-k_3}
\end{figure}

\subsection{Magnetic effects}

We present two numerical experiments, namely the Fr\'{e}edericksz transition and the magnetic field
configuration around a colloid with increasing intensity.

\subsubsection{Fr\'{e}edericksz transition}

We next study the Fr\'{e}edericksz transition \cite{freedericksz1933forces}, which is an experimental technique to determine Frank's constants $k_1,k_2,k_3$. Determining $k_4$ requires a more sophisticated experimental setup without strong anchoring, i.e.\ $\vn|_{\partial\Omega} = \vg$, so that the saddle splay term plays a role in determining minimizers \cite{allender1991determination}. We describe the set up to determine $k_1$. The domain is $\Omega = (-1,1)\times (0,w)$ and Dirichlet boundary conditions are set on the top and bottom boundaries $\Gamma_D = (-1,1)\times\{0,w\}$. For the splay configuration, the boundary condition is $\vg(x) = \ve_1$. The applied magnetic field is $\vH = H\ve_3$ for $H$ to be determined. Note that $\vn_0:=\vg$ has zero energy and is a critical point of $ E$. However, analysis in \cite{de1993physics, virga1995variational} shows that $\vn_0$ becomes unstable when
\begin{equation*}
H>\frac{1}{2w}\sqrt{\frac{k_1}{\chi_A}}.
\end{equation*}
\begin{figure}[htb]
\begin{minipage}{.4\textwidth}
\begin{center}
\includegraphics[width=1\linewidth]{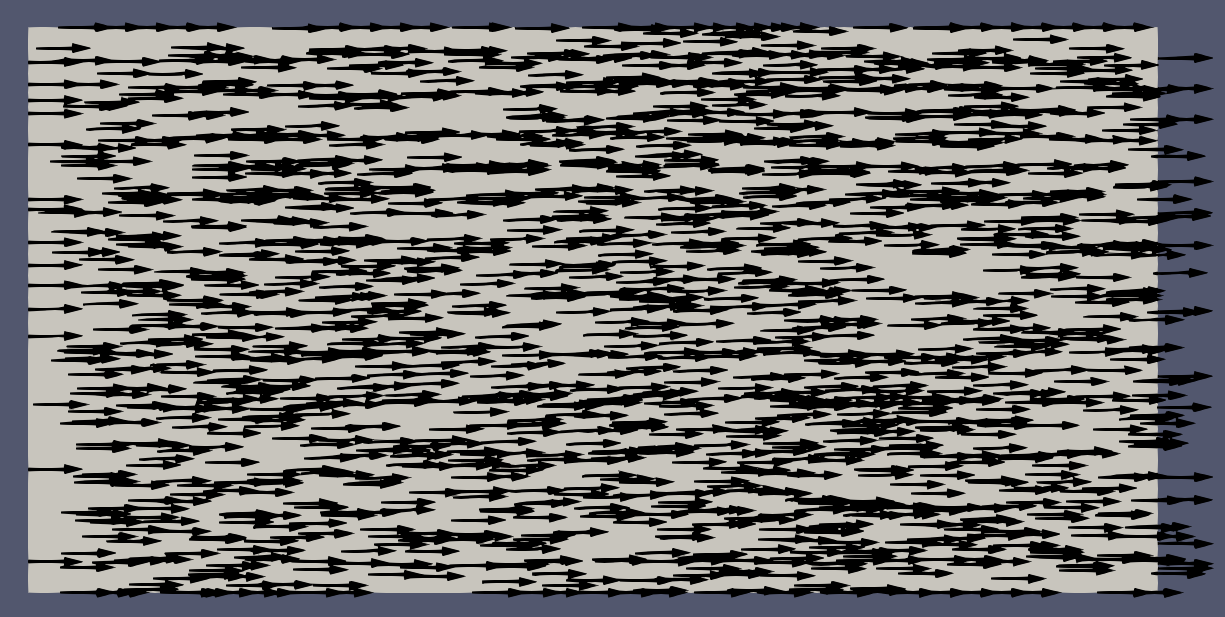}
\end{center}
\end{minipage}\hskip1.cm 
\begin{minipage}{.4\textwidth}
\begin{center}
\includegraphics[width=1\linewidth]{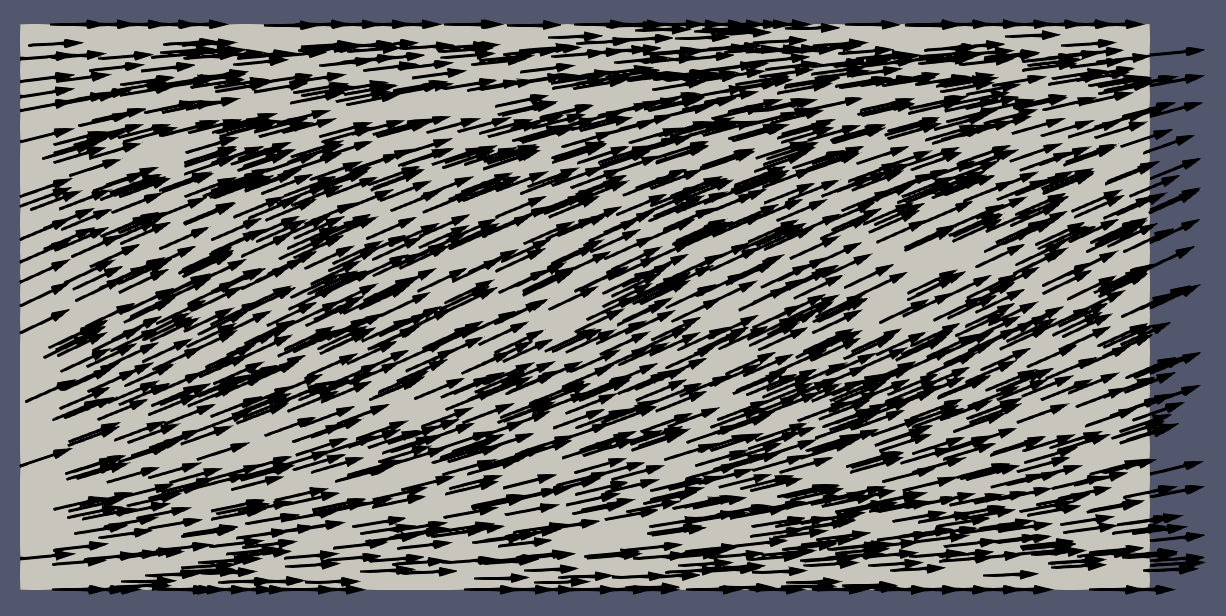}
\end{center}
\end{minipage}
\caption[Fr\'{e}edericksz transition]{Fr\'{e}edericksz transition: Initial and final configurations from Algorithm \ref{alg:grad-flow} and $k_1 = 2.3 , k_2 = 1.5, k_3 = 4.8, \chi_A = 1.21, H = 9.5$.}\label{fig:magnetic-fred}
\end{figure}
For the numerical experiment, we take the following material parameters
\[
 k_1 = 2.3 ,\; k_2 = 1.5,\; k_3 = 4.8; \; w = \frac{1}{2}, \; \chi_A = 1.21, H = 9.5,
\]
where $k_i$ are scaled constants for PAA at 125 degrees Celsius \cite[p. 123]{virga1995variational} and $\chi_A$ is the scaled constant for PAA at 122 degrees Celsius \cite[p. 174]{virga1995variational}.
The numerical parameters are
\[
h = \tau = \frac{1}{32}, \quad \varepsilon = \frac{10^{-4}}{2}.
\]
For the initial condition, we consider a perturbation of the equilibrium state
\[
\vn_h^{0} = I_h\left[\frac{\ve_1+\vu}{|\ve_1+\vu|}\right],\quad \vu = 256[x(1-x)y(1-y)]^2z(.5-z)\ve_3,
\]
with $\Vert\vu\Vert_{L^\infty(\Omega;\mathbb{R}^3)}\approx .03$. Figure \ref{fig:magnetic-fred} shows the initial and final configurations of the gradient flow. Note that Dirichlet boundary conditions were not imposed on the sides, so our use of the modified energy $ E$ may not be entirely faithful to $E_{FO}$. However we still see energy decrease in the gradient flow algorithm as evidenced by Figure \ref{fig:magnetic-energy-decay}. 

The reason why an experimenter can measure $k_1$ using this experiment is that the competition between the magnetic energy and the elastic energy only happens in the splay term. Table \ref{tab:magnetic-energy-components} shows that the splay is the dominant part of the energy that increases.

\begin{table}[h!]
	\begin{minipage}{0.5\linewidth}
\begin{tabular}{ |c |c| c| c| }\hline
& $\textsf{splay}(\vn_h)$ & $\textsf{twist}(\vn_h)$ & $\textsf{bend}(\vn_h)$ \\\hline
 Initial & 1.71e-03 & 5.15e-04 & 5.14e-04 \\ \hline
 Final & 1.82 & 2.35e-02 &    8.06e-02\\  \hline
\end{tabular}
\captionsetup{width=\linewidth}
\caption[Initial and final components of Frank energy under Fr\'{e}edericksz transition]{Fr\'{e}edericksz transition: Initial and final splay, twist, and bend for computed solution for the Fr\'{e}edericksz transition experiment with $k_1 = 2.3 , k_2 = 1.5, k_3 = 4.8, \chi_A = 1.21, H = 9.5$ and numerical parameters $h = \tau = \frac{1}{32},$ and $\varepsilon = \frac{10^{-4}}{2}.$ Note the large increase in splay relative to bend and twist indicating that most of the increase in the elastic energy is due to splay.
}\label{tab:magnetic-energy-components}
	\end{minipage}\hfill
	\begin{minipage}{.47\linewidth}
		\centering
		\includegraphics[width=\linewidth]{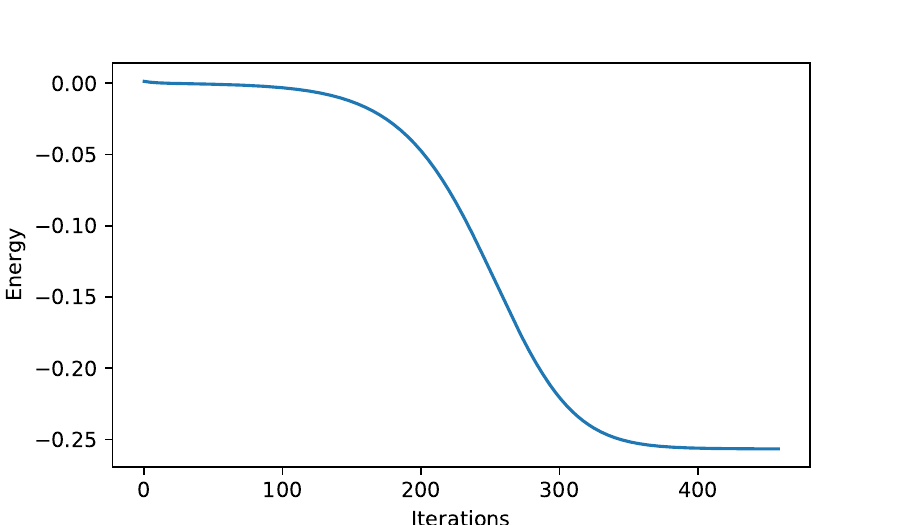}
		\captionsetup{width=\linewidth}
		\captionof{figure}{Energy decay for Fr\'{e}edericksz transition experiment vs gradient flow iterations.}
		\label{fig:magnetic-energy-decay}
	\end{minipage}
\end{table}

\subsubsection{Magnetic effects and a colloid}

This computational example reveals the influence of the magnetic field on a liquid crystal configuration around a colloid.  One salient feature of computing with a colloid is the \emph{inherent difficulty to create weakly acute meshes in 3d} due to the domain topology (see Fig.~\ref{fig:colloid}), and hence to realize projection methods that enforce energy decrease. This contrasts strikingly with the simplicity of the current projection-free approach. The setup is similar to what was done in \cite{nochetto2017finite}. The domain, boundary conditions, and magnetic field are  
\[
\Omega = [-2,2]^3\setminus B_{3/4}(0), \quad \vg(\vx) = \begin{cases}\vx/|\vx|, &\vx\in \partial B_{3/4}(0)\\ \ve_3,& \vx\in \partial [-2,2]^3\end{cases},\quad \vH = H\ve_2,
\]
where $H = 0,1,2,4$. Note that the magnetic field $\vH$ is orthogonal to the outer boundary $\vg = \ve_3$. In this sense, this setup is similar to the Fr\'{e}edericksz transition computed earlier. There is competition between matching the outer boundary condition and paying little elastic energy versus reducing the magnetic energy. The main difference with the Fr\'{e}edericksz transition is twofold: first the Dirichlet boundary condition is enforced everywhere on $\partial\Omega$, and second the presence of the colloid. The numerical parameters of Algorithm \ref{alg:grad-flow} are
\[
h=\frac{1}{8},\quad \tau =\frac{h}{4}, \quad \varepsilon = 10^{-4}.
\]

Figure \ref{fig:colloid} shows $\vn_h^{\infty}$ with varying $H$. Note that for $H=1,2$, the computed minimizer looks quite similar to the $H=0$ case. For $H=4$, the computed $\vn_h^{\infty}$ is nearly parallel to the magnetic field, except near the boundary, where Dirichlet boundary conditions are imposed. We can see that $\vn_h^{\infty}$ is nearly parallel to the magnetic field for $H=4$ since the final energy is ${E}_{total}[\vn_h^{\infty}]\approx -120$ while the final energies for $H=0,1,2$ are approximately $28.4, 27.1, 18.9$ respectively. This suggests a transition similar to Fr\'{e}ederickz occurs where the magnetic field overcomes the elastic energy.
\begin{figure}[htb]
\begin{center}
\begin{minipage}{.41\textwidth}
\begin{center}
\includegraphics[width=.9\linewidth]{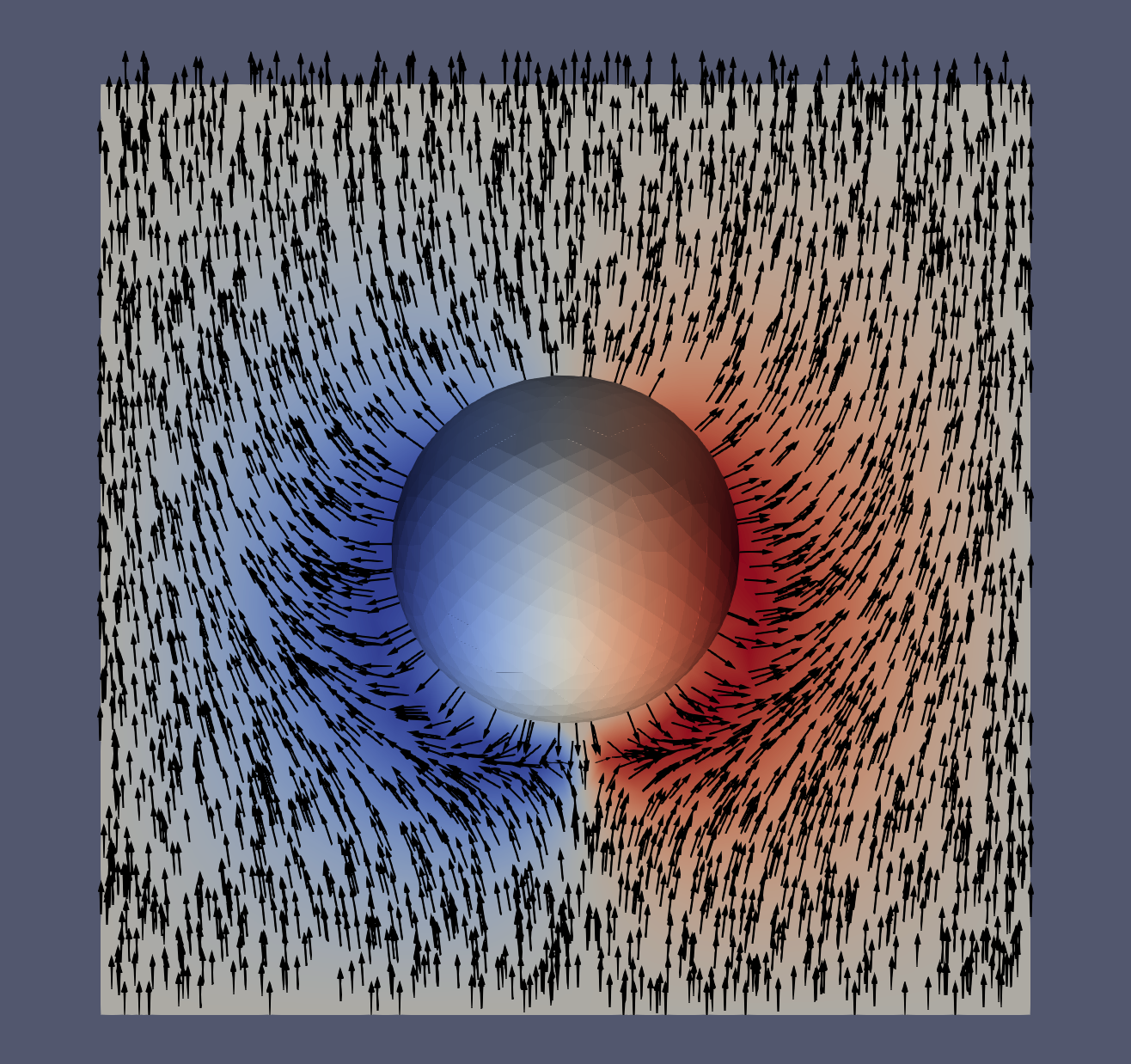}
\includegraphics[width=.9\linewidth]{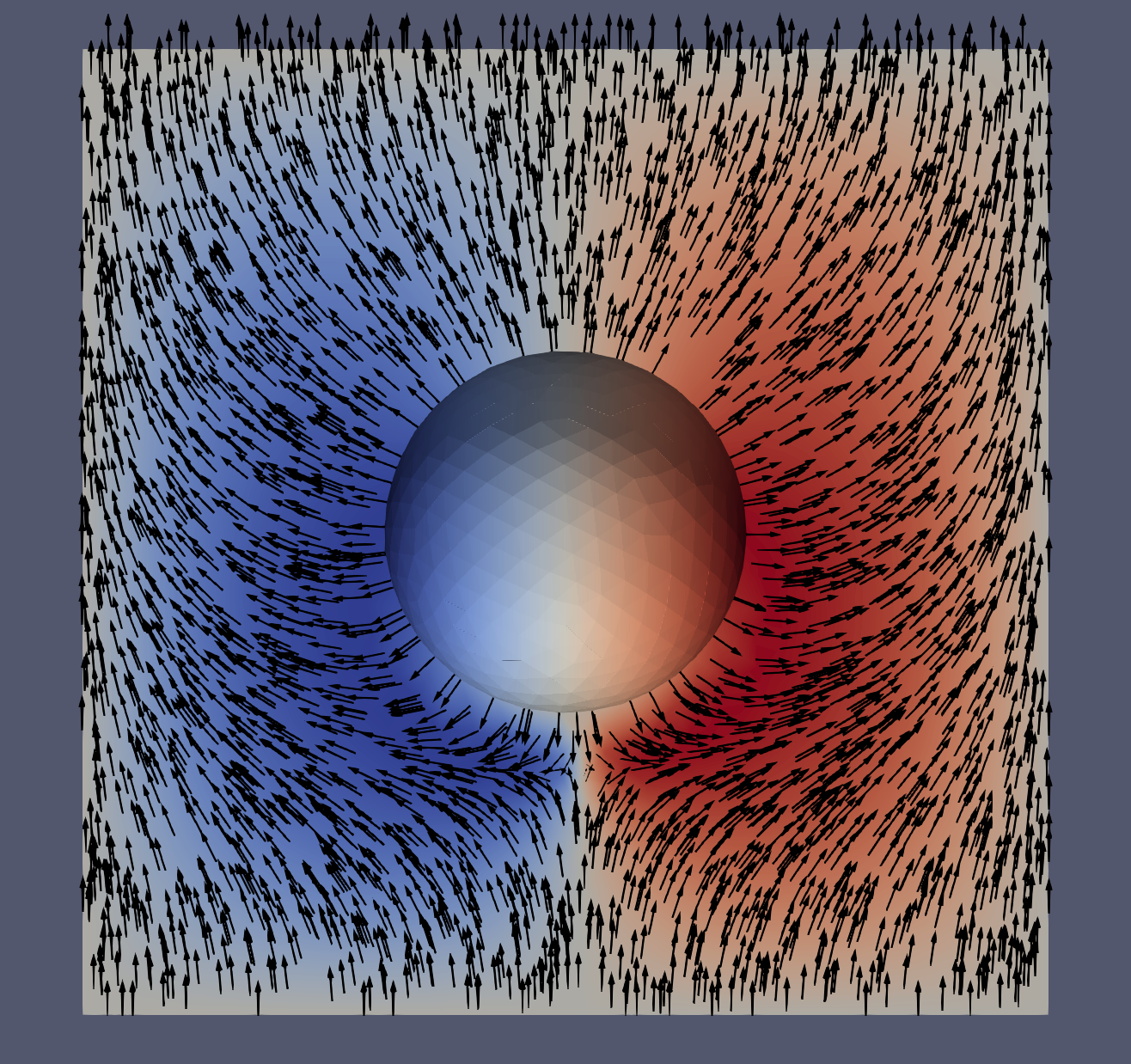}
\end{center}
\end{minipage}
\begin{minipage}{.41\textwidth}
\begin{center}
\includegraphics[width=.9\linewidth]{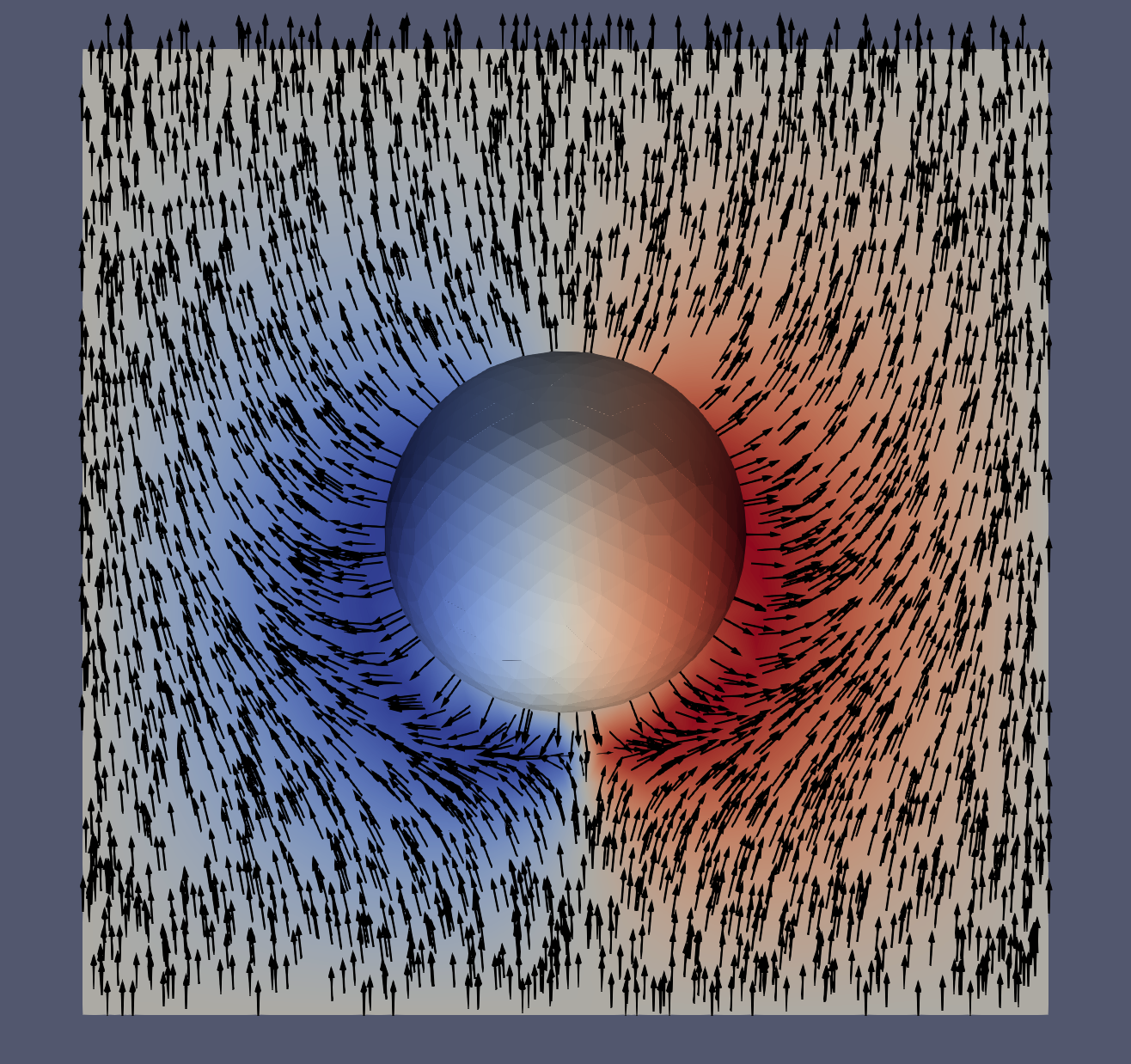}
\includegraphics[width=.9\linewidth]{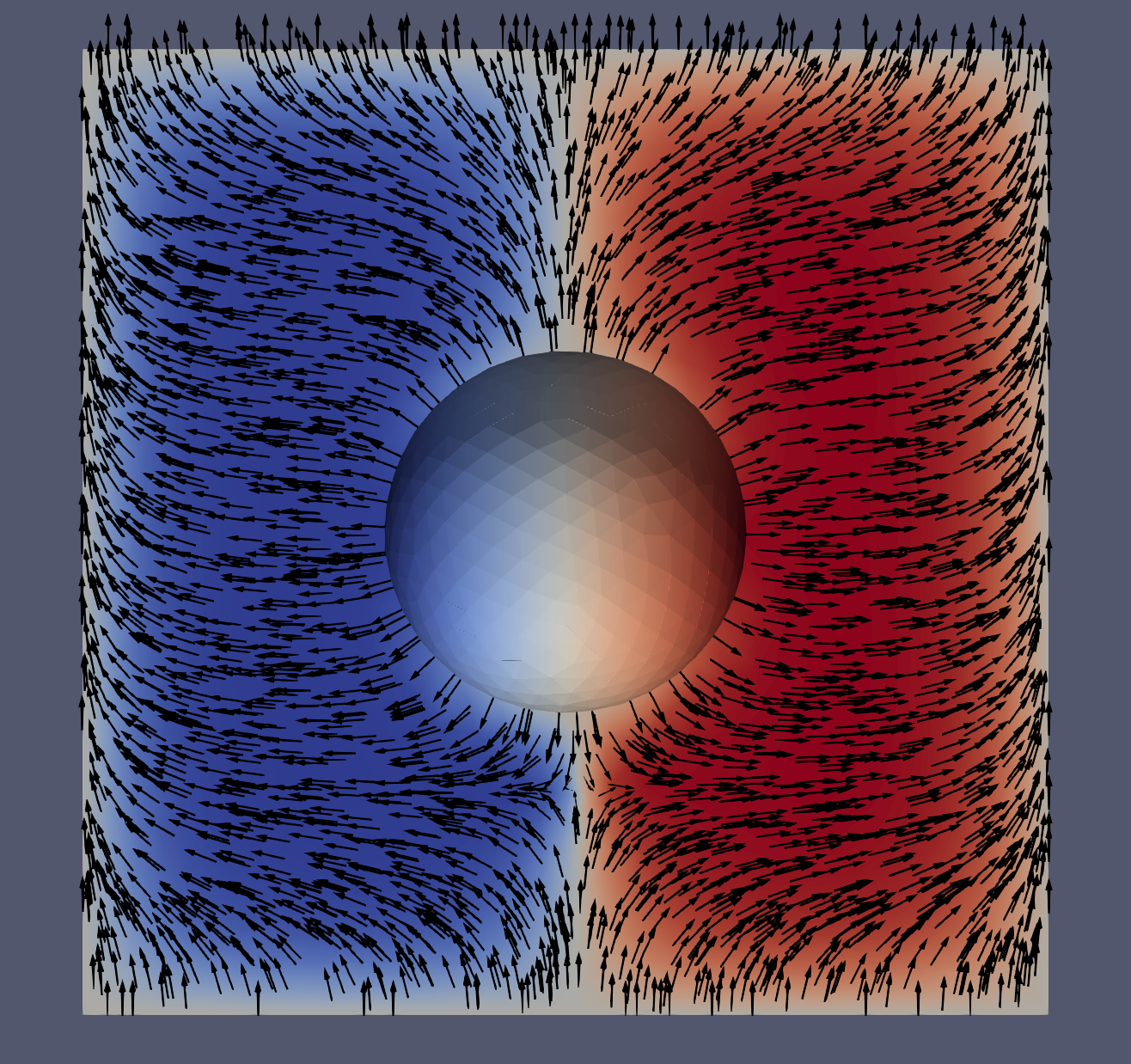}
\end{center}
\end{minipage}
\end{center}
\caption[Influence of magnetic field on liquid crystal with colloid.]{Influence of magnetic field on liquid crystal with colloid. Color corresponds to the $y$-component of $\vn_h^{\infty}$. Arrows are $\vn_h^{\infty}$ projected onto $\{x=0\}$ plane.
(Top left) $H=0$, (Top right) $H=1$, (Bottom left) $H=2$,  (Bottom right) $H=4$. 
Note the large change in behavior from $H=2$ to $H=4$. The director field for $H=2$ behaves more or less like $H=0$. However for $H=4$, the director field is almost totally parallel to $\ve_2$ except near the boundary; thereby suggesting the occurrence of a Fr\'{e}ederickz transition.}
\label{fig:colloid}
\end{figure}

\section{Conclusions}

We design and study a projection-free method for the computation of the full Frank-Oseen energy of nematic liquid crystals with Dirichlet boundary conditions.

\begin{itemize}
\item {\it $\Gamma$-convergence}: We show that our discrete problem $\Gamma$-converges to the continuous problem. The theory requires no regularity beyond $H^1(\Omega;\mathbb{R}^3)$ and allows for the presence of defects, which are of paramount importance in practice. The theory also has no restrictions on the Frank constants beyond what is required in the existence analysis and includes the effect of a fixed magnetic field.
  
\item {\it Projection-free gradient flow}: We propose a projection-free gradient flow in the spirit of \cite{bartels2016projection}, which applies to general shape-regular meshes that may not be weakly acute. Each step of the gradient flow entails solving a linear algebraic system due to the explicit treatment of the nonlinearities, which is similar to approaches in bilayer plates \cite{bartels2022stable, bonito2023gamma}. The gradient flow is energy stable, and we control the unit length constraint under the condition $\tau h^{-1}\leq \Cstab$, where $\Cstab$ depends on the Frank constants and initial data. 

\item {\it Violation of constraints}: Our  discrete admissible set only requires control over $\Vert I_h[|\vn_h|^2-1]\Vert_{L^1(\Omega)}$ and $\Vert \vn_h\Vert_{L^\infty(\Omega;\mathbb{R}^3)}$, which is easily enforced by our gradient flow algorithm. We also note that in order for the gradient flow algorithm output $\vn^\infty_h$ to satisfy $\Vert I_h[|\vn_h^\infty|^2-1]\Vert_{L^1(\Omega)}\to0$, we only need to take $\tau \approx h$, which is a mild condition.

\item {\it Computations}: We present computations on how the Frank constants influence the structure of defects. These seem to be the first such computations supported by theory. We also present computations of magnetic effects, including the interaction of a magnetic field on a liquid crystal around a colloid. This problem is notoriously difficult to assess with weakly acute meshes \cite{nochetto2017finite}.
\end{itemize}

\bibliographystyle{plain}
\bibliography{references}

\begin{thebibliography}{10}

\bibitem{adler2015energy}
James~H Adler, Timothy~J Atherton, DB~Emerson, and Scott~P MacLachlan.
\newblock An energy-minimization finite-element approach for the frank--oseen
  model of nematic liquid crystals.
\newblock {\em SIAM Journal on Numerical Analysis}, 53(5):2226--2254, 2015.

\bibitem{ahrens2005paraview}
James Ahrens, Berk Geveci, and Charles Law.
\newblock Paraview: An end-user tool for large data visualization.
\newblock {\em The visualization handbook}, 717(8), 2005.

\bibitem{allender1991determination}
David~W Allender, GP~Crawford, and JW~Doane.
\newblock Determination of the liquid-crystal surface elastic constant k 24.
\newblock {\em Physical review letters}, 67(11):1442, 1991.

\bibitem{alouges1997new}
Fran{\c{c}}ois Alouges.
\newblock A new algorithm for computing liquid crystal stable configurations:
  the harmonic mapping case.
\newblock {\em SIAM journal on numerical analysis}, 34(5):1708--1726, 1997.

\bibitem{alouges1997minimizing}
Fran{\c{c}}ois Alouges and Jean-Michel Ghidaglia.
\newblock Minimizing oseen-frank energy for nematic liquid crystals: algorithms
  and numerical results.
\newblock In {\em Annales de l'IHP Physique th{\'e}orique}, volume~66, pages
  411--447, 1997.

\bibitem{antil2021approximation}
Harbir Antil, S{\"o}ren Bartels, and Armin Schikorra.
\newblock Approximation of fractional harmonic maps.
\newblock {\em arXiv preprint arXiv:2104.10049}, 2021.

\bibitem{bartels2005stability}
S{\"o}ren Bartels.
\newblock Stability and convergence of finite-element approximation schemes for
  harmonic maps.
\newblock {\em SIAM journal on numerical analysis}, 43(1):220--238, 2005.

\bibitem{bartels2015numerical}
S{\"o}ren Bartels.
\newblock {\em Numerical methods for nonlinear partial differential equations},
  volume~47.
\newblock Springer, 2015.

\bibitem{bartels2016projection}
S{\"o}ren Bartels.
\newblock Projection-free approximation of geometrically constrained partial
  differential equations.
\newblock {\em Mathematics of Computation}, 85(299):1033--1049, 2016.

\bibitem{bartels2022benchmarking}
S{\"o}ren Bartels, Klaus B{\"o}hnlein, Christian Palus, and Oliver Sander.
\newblock Benchmarking numerical algorithms for harmonic maps into the sphere.
\newblock {\em arXiv preprint arXiv:2209.13665}, 2022.

\bibitem{bartels2022stable}
S{\"o}ren Bartels and Christian Palus.
\newblock Stable gradient flow discretizations for simulating bilayer plate
  bending with isometry and obstacle constraints.
\newblock {\em IMA Journal of Numerical Analysis}, 42(3):1903--1928, 2022.

\bibitem{bartels2022quasi}
S{\"o}ren Bartels, Christian Palus, and Zhangxian Wang.
\newblock Quasi-optimal error estimates for the approximation of stable
  harmonic maps.
\newblock {\em arXiv preprint arXiv:2209.11985}, 2022.

\bibitem{bonito2023gamma}
Andrea Bonito, Ricardo~H Nochetto, and Shuo Yang.
\newblock {$\Gamma$}-convergent {LDG} method for large bending deformations of
  bilayer plates.
\newblock {\em arXiv preprint arXiv:2301.03151}, 2023.

\bibitem{borthagaray2020structure}
Juan~Pablo Borthagaray, Ricardo~H Nochetto, and Shawn~W Walker.
\newblock A structure-preserving fem for the uniaxially constrained q-tensor
  model of nematic liquid crystals.
\newblock {\em Numerische Mathematik}, 145(4):837--881, 2020.

\bibitem{chen1989weak}
Yunmei Chen.
\newblock The weak solutions to the evolution problems of harmonic maps.
\newblock {\em Mathematische Zeitschrift}, 201(1):69--74, 1989.

\bibitem{cohen1987minimum}
Robert Cohen, Robert Hardt, David Kinderlehrer, San-Yin Lin, and Mitchell
  Luskin.
\newblock Minimum energy configurations for liquid crystals: Computational
  results.
\newblock In {\em Theory and Applications of Liquid Crystals}, pages 99--121.
  Springer, 1987.

\bibitem{cohen1989relaxation}
Robert Cohen, San-Yih Lin, and Mitchell Luskin.
\newblock Relaxation and gradient methods for molecular orientation in liquid
  crystals.
\newblock {\em Computer Physics Communications}, 53(1-3):455--465, 1989.

\bibitem{davis1998finite}
Timothy~A Davis and Eugene~C Gartland~Jr.
\newblock Finite element analysis of the landau--de gennes minimization problem
  for liquid crystals.
\newblock {\em SIAM Journal on Numerical Analysis}, 35(1):336--362, 1998.

\bibitem{de1993physics}
Pierre-Gilles De~Gennes and Jacques Prost.
\newblock {\em The physics of liquid crystals}.
\newblock Number~83. Oxford university press, 1993.

\bibitem{ericksen1962nilpotent}
JL~Ericksen.
\newblock Nilpotent energies in liquid crystal theory.
\newblock {\em Archive for Rational Mechanics and Analysis}, 10(1):189--196,
  1962.

\bibitem{ericksen1966}
JL~Ericksen.
\newblock Inequalities in liquid crystal theory.
\newblock {\em The physics of Fluids}, 9(6):1205--1207, 1966.

\bibitem{ericksen1991liquid}
JL~Ericksen.
\newblock Liquid crystals with variable degree of orientation.
\newblock {\em Archive for Rational Mechanics and Analysis}, 113(2):97--120,
  1991.

\bibitem{frank1958liquid}
Frederick~C Frank.
\newblock I. liquid crystals. on the theory of liquid crystals.
\newblock {\em Discussions of the Faraday Society}, 25:19--28, 1958.

\bibitem{freedericksz1933forces}
Vsevolod Fr{\'e}edericksz and V~Zolina.
\newblock Forces causing the orientation of an anisotropic liquid.
\newblock {\em Transactions of the Faraday Society}, 29(140):919--930, 1933.

\bibitem{glowinski2003operator}
Roland Glowinski, Ping Lin, and X-B Pan.
\newblock An operator-splitting method for a liquid crystal model.
\newblock {\em Computer physics communications}, 152(3):242--252, 2003.

\bibitem{hardt1987mathematical}
Robert Hardt and David Kinderlehrer.
\newblock Mathematical questions of liquid crystal theory.
\newblock In {\em Theory and applications of liquid crystals}, pages 151--184.
  Springer, 1987.

\bibitem{hardt1986existence}
Robert Hardt, David Kinderlehrer, and Fang-Hua Lin.
\newblock Existence and partial regularity of static liquid crystal
  configurations.
\newblock {\em Communications in mathematical physics}, 105(4):547--570, 1986.

\bibitem{hu2009saddle}
Qiya Hu, Xue-Cheng Tai, and Ragnar Winther.
\newblock A saddle point approach to the computation of harmonic maps.
\newblock {\em SIAM Journal on Numerical Analysis}, 47(2):1500--1523, 2009.

\bibitem{hu2014newton}
Qiya Hu and Long Yuan.
\newblock A newton-penalty method for a simplified liquid crystal model.
\newblock {\em Advances in Computational Mathematics}, 40(1):201--244, 2014.

\bibitem{kim2020tunable}
MinSu Kim and Francesca Serra.
\newblock Tunable dynamic topological defect pattern formation in nematic
  liquid crystals.
\newblock {\em Advanced Optical Materials}, 8(1):1900991, 2020.

\bibitem{kinderlehrer1992second}
David Kinderlehrer and Biao Ou.
\newblock Second variation of liquid crystal energy at $x/| x|$.
\newblock {\em Proceedings of the Royal Society of London. Series A:
  Mathematical and Physical Sciences}, 437(1900):475--487, 1992.

\bibitem{nochetto2022gamma}
Ricardo~H Nochetto, Michele Ruggeri, and Shuo Yang.
\newblock Gamma-convergent projection-free finite element methods for nematic
  liquid crystals: The ericksen model.
\newblock {\em SIAM Journal on Numerical Analysis}, 60(2):856--887, 2022.

\bibitem{nochetto2017finite}
Ricardo~H Nochetto, Shawn~W Walker, and Wujun Zhang.
\newblock A finite element method for nematic liquid crystals with variable
  degree of orientation.
\newblock {\em SIAM Journal on Numerical Analysis}, 55(3):1357--1386, 2017.

\bibitem{oseen1933theory}
Carl~W. Oseen.
\newblock The theory of liquid crystals.
\newblock {\em Transactions of the Faraday Society}, 29(140):883--899, 1933.

\bibitem{paige1975solution}
Christopher~C. Paige and Michael~A. Saunders.
\newblock Solution of sparse indefinite systems of linear equations.
\newblock {\em SIAM journal on numerical analysis}, 12(4):617--629, 1975.

\bibitem{schadt1997liquid}
Martin Schadt.
\newblock Liquid crystal materials and liquid crystal displays.
\newblock {\em Annual review of materials science}, 27(1):305--379, 1997.

\bibitem{schimming2021numerical}
Cody~D Schimming, Jorge Vi{\~n}als, and Shawn~W Walker.
\newblock Numerical method for the equilibrium configurations of a maier-saupe
  bulk potential in a q-tensor model of an anisotropic nematic liquid crystal.
\newblock {\em Journal of Computational Physics}, 441:110441, 2021.

\bibitem{schoberl2014c++}
Joachim Sch{\"o}berl.
\newblock C++ 11 implementation of finite elements in ngsolve.
\newblock {\em Institute for Analysis and Scientific Computing, Vienna
  University of Technology}, 30, 2014.

\bibitem{virga1995variational}
Epifanio~G Virga.
\newblock {\em Variational theories for liquid crystals}, volume~8.
\newblock CRC Press, 1995.

\bibitem{walker2020finite}
Shawn~W Walker.
\newblock A finite element method for the generalized ericksen model of nematic
  liquid crystals.
\newblock {\em ESAIM: Mathematical Modelling and Numerical Analysis},
  54(4):1181--1220, 2020.

\bibitem{wang2021modelling}
Wei Wang, Lei Zhang, and Pingwen Zhang.
\newblock Modelling and computation of liquid crystals.
\newblock {\em Acta Numerica}, 30:765--851, 2021.

\bibitem{xia2021augmented}
Jingmin Xia, Patrick~E Farrell, and Florian Wechsung.
\newblock Augmented lagrangian preconditioners for the oseen--frank model of
  nematic and cholesteric liquid crystals.
\newblock {\em BIT Numerical Mathematics}, 61(2):607--644, 2021.

\end{thebibliography}

\end{document}